\documentclass[a4paper, 11pt, final]{article} 

\usepackage{amssymb, amsmath, amsthm, mathrsfs}
\usepackage[notref,notcite]{showkeys}
\usepackage{enumitem}
\usepackage{a4wide}
\usepackage{hyperref}
\usepackage{array}

\usepackage{comment}
\usepackage{pgf,tikz}
\usepackage{pgfplots}
\usepackage{float}

\numberwithin{equation}{section}
\allowdisplaybreaks[4]

\newtheorem{proposition}{Proposition}[section]
\newtheorem{theorem}[proposition]{Theorem}
\newtheorem{lemma}[proposition]{Lemma}
\newtheorem{definition}[proposition]{Definition}

\renewenvironment{proof}{\smallskip\noindent\emph{\textbf{Proof.}}%
  \hspace{1pt}}{\hspace{-5pt}{\nobreak\quad\nobreak\hfill\nobreak%
    $\square$\vspace{2pt}\par}\smallskip\goodbreak}

\newenvironment{proofof}[1]{\smallskip\noindent{\textbf{Proof~of~#1.}}%
  \hspace{1pt}}{\hspace{-5pt}{\nobreak\quad\nobreak\hfill\nobreak%
    $\square$\vspace{2pt}\par}\smallskip\goodbreak}

\newcommand{\C}[1]{\mathbf{C}^{#1}}
\newcommand{\Cc}[1]{\mathbf{C}_c^{#1}}

\newcommand{\BV}{\mathbf{BV}}

\newcommand{\PC}{\mathbf{PC}}

\renewcommand{\L}[1]{{\mathbf{L}^#1}}
\newcommand{\Lloc}[1]{{\mathbf{L}_{\mathbf{loc}}^{#1}}}

\newcommand{\modulo}[1]{{\left|#1\right|}}
\newcommand{\norma}[1]{{\left\|#1\right\|}}
\newcommand{\caratt}[1]{{\chi_{\strut#1}}}
\newcommand{\reali}{{\mathbb{R}}}

\renewcommand{\epsilon}{\varepsilon}
\renewcommand{\phi}{\varphi}
\renewcommand{\theta}{\vartheta}
\renewcommand{\O}{\mathcal{O}(1)}
\newcommand{\tv}{\mathinner{\rm TV}}

\newcommand{\Lip}{\mathinner{Lip}}
\renewcommand{\d}[1]{\mathinner{\mathrm{d}{#1}}}

\makeatletter \let\@fnsymbol\@arabic \makeatother

\title{Well Posedness and Characterization of Solutions to \\
  Non Conservative Products in Non Homogeneous \\
  Fluid Dynamics Equations}

\author{Rinaldo M. Colombo\footnotemark[1] \and Graziano
  Guerra\footnotemark[2] \and Yannick Holle\footnotemark[3]}

\begin{document}

\maketitle

\footnotetext[1]{INdAM Unit \& Department of Information Engineering,
  University of Brescia, via Branze, 38, 25123 Brescia,
  Italy. e--mail: \texttt{rinaldo.colombo@unibs.it}. ORCID:
  \texttt{0000-0003-0459-585X}.} \footnotetext[2]{Universit\`a degli
  Studi di Milano Bicocca, Dipartimento di Matematica e Applicazioni,
  via R.~Cozzi, 55, 20126 Milano, Italy. e--mail:
  \texttt{graziano.guerra@unimib.it}. ORCID:
  \texttt{0000-0003-2615-2750}.} \footnotetext[3]{RWTH Aachen
  University, Institut f\"ur Mathematik, Templergraben 55, 52062
  Aachen, Germany. e--mail: \texttt{holle@eddy.rwth-aachen.de}. ORCID:
  \texttt{0000-0002-7841-8525}.\vspace{0.05cm}}

\begin{abstract}
  \noindent Consider a balance law where the flux depends explicitly
  on the space variable. At jump discontinuities, modeling
  considerations may impose the defect in the conservation of some
  quantities, thus leading to non conservative products. Below, we
  deduce the evolution in the smooth case from the jump conditions at
  discontinuities. Moreover, the resulting framework enjoys well
  posedness and solutions are uniquely characterized. These results
  apply, for instance, to the flow of water in a canal with varying
  width and depth, as well as to the inviscid Euler equations in pipes
  with varying geometry.

  \medskip


\end{abstract}

\section{Introduction}
\label{sec:Intro}

The flow of water in a canal of smoothly varying width and smoothly
varying bed elevation is described by the following balance law
\begin{equation}
  \label{eq:53}
  \left\{
    \begin{array}{l}
      \partial_t a + \partial_x q = 0
      \\
      \partial_t q + \partial_x \left( \dfrac{q^2}{a}
      + \dfrac12 \, g \, \dfrac{a^2}\sigma\right)
      =
      \dfrac12 \, g\, \dfrac{a^2}{\sigma^2} \, \partial_x \sigma
      -
      g \, a \, \partial_x b \,,
    \end{array}
  \right.
\end{equation}
see~\cite[Formula~(1.1)] {MR4179858}. Here $g$ is gravity, $t$ is
time, $x$ is the longitudinal coordinate along the canal,
$a = a (t,x)$ is the wetted cross sectional area, $q = q (t,x)$ is the
water flow, $\sigma = \sigma (x)$ is the canal width and $b = b (x)$
is the height of the bottom.

The presence of discontinuities in the channel width $\sigma$ or in
the bed elevation $b$ prevents the application of standard theorems
to~\eqref{eq:53}. Indeed, discontinuities arise in the flux and non
conservative products appear in the source term. As is well know the
latter terms lack a unique way to be defined. As a reference to non
conservative products, we refer to~\cite{DalMasoMurat1995, MR1718304}.

In the present work, we construct a framework where~\eqref{eq:53} has
a meaning and is well posed, requiring $\sigma$ and $b$ to be merely
of bounded variation.

Whenever $\sigma$ and $b$ are piecewise constant with jumps at, say,
$\bar x_1, \ldots, \bar x_N$, equation~\eqref{eq:53} fits into the
non--homogeneous system of conservation laws
\begin{displaymath}
  \partial_t u + \partial_x f\left(\zeta (x),u\right) = 0
  \qquad
  x \in \reali \setminus \{\bar x_1, \ldots, \bar x_N\} \,,
\end{displaymath}
equipped with suitable conditions
\begin{equation}
  \label{eq:IntroSingleCoupling}
  \Psi
  \left(
    \zeta (\bar x_i+),u(t,\bar x_i+),
    \zeta (\bar x_i-),u(t,\bar x_i-)
  \right) = 0
  \quad \mbox{for a.e. }t>0
  \mbox{ and } i=1, \ldots, N
\end{equation}
where $\zeta$ is as in~\eqref{eq:77}.

This junction condition, thanks to to the assumptions below,
by~\cite[Lemma~4.1]{ColomboGuerraHolle} and by an immediate extension
of~\cite[Lemma~4.2]{ColomboGuerraHolle}, can be reformulated as
\begin{equation}
  \label{eq:DefPsiCoupling}
  f\left(\zeta (\bar x+), u(t,\bar x+)\right)
  -
  f\left(\zeta (\bar x-),  u(t,\bar x-)\right)
  =
  \Xi\big(\zeta (\bar x+),\zeta (\bar x-),u(t, \bar x-)\big)
  \quad \mbox{for a.e. }t>0
\end{equation}
where $\bar x$ is any point of jump and $\Xi$ measures the defect in
the conservation of $u$ at $\bar x$.

We show that choosing~\eqref{eq:DefPsiCoupling} actually singles out
the source term in~\eqref{eq:21} below, which accounts both for the
smooth changes as well as for the points of jump in $\zeta$. In the
case of~\eqref{eq:53}, this amounts to show that a careful choice of
$\Xi$ allows to extend~\eqref{eq:53} to the case of $\sigma$ and $b$
in $\BV$.

More precisely, when $\zeta \in \BV (\reali; \reali^p)$ and given a
piecewise constant approximation $\zeta^h$ of $\zeta$ with finite
number of jumps located at $\bar x \in \mathcal{I} (\zeta^h)$, we
obtain the following balance law with measure-valued source term
\begin{equation}
  \label{eq:68}
  \left\{
    \begin{array}{l}
      \partial_t u+\partial_x f(\zeta^h , u)
      =
      \sum\limits_{\bar x \in \mathcal{I} (\zeta^h)}
      \Xi \left(\zeta^h ( \bar x+), \zeta^h ( \bar x-), u (\cdot, \bar  x-)\right) \;
      \delta_{\bar  x}
      \\
      u (0,x) = u_o (x),
    \end{array}
  \right.
\end{equation}
where $\delta_{\bar x}$ denotes the Dirac measure at $\bar x$.

In the general - non characteristic - setting established below,
solutions to~\eqref{eq:68} are shown to converge as $\zeta^h$
converges to $\zeta$ in a suitable - strong - sense, to solutions to
\begin{equation}
  \label{eq:21}
  \!\!\!\!\!\!\!\!
  \left\{
    \begin{array}{@{}l@{}}
      \partial_t u
      +
      \partial_x f(\zeta,u)
      =
      \sum\limits_{\bar x \in \mathcal I (\zeta)}
      \Xi\left(\zeta (\bar x+), \zeta (\bar x-), u (\cdot,\bar x-)\right)
      \delta_{\bar x}
      +
      D^+_{v} \Xi (\zeta, \zeta, u ) \,
      \norma{\mu}
      \\
      u (0,x)
      =
      u_o (x) \,.
    \end{array}
  \right.
\end{equation}
The terms in the singular source term above are defined as
follows. Since $\zeta \in \BV (\reali; \reali^p)$, the right and left
limits $\zeta (\bar x+)$ and $\zeta (\bar x-)$ are well defined and
the distributional derivative $D\zeta$ can be split in a discrete part
and a non discrete one, which may contain a Cantor part:
\begin{equation}
  \label{eq:20}
  D\zeta
  =
  \sum_{\bar x \in \mathcal{I} (\zeta)} \left(\zeta (\bar x+) - \zeta (\bar x-)\right) \, \delta_{\bar x}
  +
  v \, \norma{\mu} \,,
\end{equation}
where the function $v$ is Borel measurable with norm $1$ and $\mu$ is
the non atomic part of $D\zeta$.  In~\eqref{eq:21} we also used the
(one sided) directional derivative
\begin{equation}
  \label{eq:26}
  D^+_{v} \Xi (z, z, u)
  =
  \lim_{t \to 0+}
  \dfrac{
    \Xi(z + t \, v, z, u)
    -
    \Xi(z, z, u)}{t} \,.
\end{equation}

A preliminary result was obtained in~\cite{ColomboGuerraHolle}, where
a sequence of solutions to~\eqref{eq:68} is shown to converge to a
solution to~\eqref{eq:21}. Here, we extend the framework
in~\cite{ColomboGuerraHolle} considering space dependent fluxes, prove
that~\eqref{eq:68} generates a Lipschitz semigroup, say $S^h$, and
show the convergence of $S^h$ to a semigroup whose orbits
solve~\eqref{eq:21}. Moreover, we provide a full characterization of
the solutions to~\eqref{eq:21} in terms of integral inequalities, in
the spirit of~\cite{Bressan}.

\medskip

The present results comprise the case of balance laws with a space
dependent flux and a non conservative source term of the type
\begin{equation}
  \label{eq:52}
  \partial_t u + \partial_x f (\zeta,u) = D_\zeta G (\zeta,u) \, D\zeta
\end{equation}
see~\cite[\S~3.4]{ColomboGuerraHolle}. Setting $p=2$,
$\mathcal{Z} = \mathopen]0, +\infty\mathclose[ \times \reali$ and
\begin{equation}
  \label{eq:77}
  \zeta (x) = \left[
    \begin{array}{c}
      \left.1 \middle/\sigma (x)\right.
      \\
      b (x)
    \end{array}
  \right]
  \quad \mbox{ and } \quad
  G \left(z,(a,q)\right)
  =
  \left[
    \begin{array}{c}
      0
      \\
      -\frac12 \, g \, a^2 \, z_1 - g \, a \, z_2
    \end{array}
  \right]
\end{equation}
we see that~\eqref{eq:53} fits into~\eqref{eq:52}:
\begin{equation}
  \label{eq:78}
  \left\{
    \begin{array}{l}
      \partial_t a + \partial_x q = 0
      \\
      \partial_t q + \partial_x \left( \dfrac{q^2}{a}
      + \dfrac12 \, g \, \zeta_1 \, a^2\right)
      =
      -\dfrac12 \, g\, a^2 \, \partial_x \zeta_1
      -
      g \, a \, \partial_x \zeta_2
    \end{array}
  \right.
\end{equation}
and hence our main result, Theorem~\ref{thm:sgrp}, applies setting,
for instance,
\begin{displaymath}
  \Xi(z^+, z^-, u^-) = G(z^+, u^-) - G(z^-,u^-) \,.
\end{displaymath}
As noted in~\cite[Section~3]{ColomboGuerraHolle}, different choices of
$\Xi$ may yield different solutions emanating from discontinuities in
$\zeta$ while giving the same solutions wherever $\zeta$ is smooth.

\medskip

Moreover, all the applications considered
in~\cite[Section~3]{ColomboGuerraHolle} fall within the scope of
Theorem~\ref{thm:sgrp}. They are the classical $p$-system, i.e.,
isentropic gas dynamics, in a pipe with varying section or with bends,
see also~\cite{MR2041456}, as well as the full Euler compressible
system in pipes, see also~\cite{MR2177892}.

Thus, in addition to the existence of solutions proved
in~\cite{ColomboGuerraHolle}, here we also ensure the Lipschitz
continuous dependence of the solutions on the initial data. Further,
we provide a characterization of the solutions by means of the
integral relations \emph{(i)} and~\emph{(ii)} in
Theorem~\ref{thm:sgrp}. These results hold under assumptions on the
source terms that are strictly weaker than those
in~\cite{AmadoriGosseGuerra2002}. Moreover, the present construction
encompasses fluxes explicitly depending on the space variable.

\section{Hypotheses and Main Theorem}
\label{sec:General}

Here, for a real number $x$, $\modulo{x}$ is its absolute value, while
$\norma{v}$ is the Euclidean norm of a vector $v$ and $\norma{\mu}$ is
the total variation of a measure $\mu$. The open ball in $\reali^n$
centered at $u$ with radius $\delta$ is denoted by $B(u;\delta)$, its
closure is $\overline{B (u; \delta)}$. We also use the following
notation for left/right limits and for differences at a point:
\begin{displaymath}
  F(x-)  =  \lim_{\xi\to x^{-}}F(\xi)\, ,\quad
  F(x+)  =  \lim_{\xi\to x^{+}}F(\xi)
  \quad \mbox{ and } \quad
  \Delta F(x)  =  F(x+)-F(x-) \,.
\end{displaymath}
Throughout, we choose the \emph{left}--continuous representatives of
$\BV$ functions.

The problem we tackle is defined by the flow $f$ and by the functions
$\Xi$ and $\zeta$. Here we detail the key assumptions, $\Omega$ being
an open convex subset of $\reali^n$ and $\mathcal{Z}$ a convex open
subset of $\reali^p$:

\begin{enumerate}[label={\bf{(f.\arabic*)}}]
\item \label{it:f1} $f\in \C4(\mathcal{Z}\times\Omega; \reali^n)$;
\item \label{it:f2} the Jacobian matrix $D_u f (z,u)$ is strictly
  hyperbolic for every $ z \in \mathcal{Z}$ and $u \in \Omega$;
\item \label{it:f3} each characteristic field is either genuinely
  nonlinear or linearly degenerate for all $ z \in \mathcal{Z}$.
\end{enumerate}
\noindent In the latter assumption we refer to the classical definitions
by Lax~\cite{Lax1957}, see also~\cite[\S~7.5]{Dafermos2000}.

By~\ref{it:f1} and~\ref{it:f2} we know that, possibly restricting
$\Omega$, the eigenvalues $\lambda_1(z,u)$, $\dots$ ,$\lambda_n(z,u)$
of $D_uf(z,u)$ depend smoothly on $z$ and can be indexed so that, for
all $u \in \Omega$ and $z \in \mathcal{Z}$,
\begin{displaymath}
  \lambda_1(z,u) < \lambda_2(z,u) < \dots < \lambda_n(z,u) \,.
\end{displaymath}
We thus require the usual non resonance condition
\begin{enumerate}[label={\bf{(f.4)}}]
\item \label{it:f4} there exists $i_o \in \{1,\dots,n-1\}$ such that
  $\lambda_{i_o}(z,u)<0<\lambda_{i_o+1}(z,u)$ for all
  $z \in \mathcal{Z}$ and all $u \in \Omega$.
\end{enumerate}
\noindent Note that both the cases of characteristic speeds being
either all positive or all negative are simpler.

On the function $\Xi$ in~\eqref{eq:DefPsiCoupling}, used to rewrite
the coupling condition induced by $\Psi$, we require:

\begin{enumerate}[label={\bf{($\mathbf{\Xi}$.\arabic*)}}]
\item \label{eq:Xi1}
  $\Xi \colon \mathcal{Z}\times\mathcal{Z}\to \C1(\overline{\Omega};
  \reali^n)$ is a Lipschitz continuous map and
  $\Xi \colon \mathcal{Z}\times\mathcal{Z}\to \C2(\overline{\Omega};
  \reali^n)$;
\item \label{eq:Xi2}
  $\sup_{z^+,z^-\in\mathcal Z}\norma{\Xi(z^+, z^-,
    \cdot)}_{\C{2}(\Omega; \reali)}< +\infty$;
\item \label{eq:Xi3} $\Xi(z, z, u) = 0$ for every $z \in \mathcal{Z}$
  and $u \in \Omega$;
\item \label{eq:Xi4} there exists a non decreasing map
  $\sigma \colon \mathopen[0, \bar t\mathclose[ \to \reali$ with
  $\lim_{t\to 0}\sigma(t)=0$ such that for all
  $(z,v,u) \in \mathcal{Z} \times \overline{B(0;1)} \times \Omega$
  \begin{displaymath}
    \norma{\Xi (z + t \, v, z, u) - D^+_v\Xi (z,z,u) \, t}
    \leq
    \sigma (t) \, t
  \end{displaymath}
  and moreover the map $(z,v,u) \to D^+_v \Xi (z,z,u)$ is Lipschitz
  continuous.
\end{enumerate}
\noindent In the latter condition, recall the definition~\eqref{eq:26}
of the Dini right derivative. Our requiring this low regularity,
i.e.~the mere existence of the Dini derivative rather than
differentiability, is motivated by the example of a pipe with angles,
where $\Xi$ depends on $\norma{z^+ - z^-}$,
see~\cite[Section~3.1]{ColomboGuerraHolle}.

In Problem~\eqref{eq:21} we require that
$\zeta \in \BV (\reali; \mathcal{Z})$. Throughout, the map $\zeta$ is
assumed to be left continuous and the set of jump discontinuities in
$\zeta$ is denoted by $\mathcal{I} (\zeta)$, with
$\mathcal{I} (\zeta) \subset \reali$.

We now precisely state what we mean by \emph{solution}
to~\eqref{eq:21}.

\begin{definition}
  \label{def:sol}
  Let $u_o \in \Lloc1 (\reali; \reali^n)$.  A map
  $u \in \C0 (\mathopen[0, +\infty\mathclose[; \Lloc1 (\reali;
  \reali^n))$ with $u (t) \in \BV (\reali; \reali^n)$ and left
  continuous for all $t \in \reali_+$, is a \emph{solution}
  to~\eqref{eq:21} if for all test functions
  $\phi \in \Cc1 (\mathopen]0, +\infty\mathclose[ \times \reali;
  \reali)$,
  \begin{eqnarray}
    \nonumber
    &
    & -\int_0^{+\infty} \int_{\reali}
      \left(
      u (t,x) \, \partial_t \phi (t,x)
      +
      f\left(\zeta (x), u (t,x)\right) \, \partial_x \phi (t,x)
      \right)
      \d{x} \d{t}
    \\
    \label{eq:48}
    & =
    & \sum_{{\bar x} \in \mathcal{I} (\zeta)}
      \int_0^{+\infty}
      \Xi\left(\zeta ({\bar x}+), \zeta ({\bar x}), u (t, {\bar x})\right)
      \phi (t,{\bar x}) \d{t}
    \\
    \nonumber
    &
    & +
      \int_0^{+\infty} \int_{\reali}
      D_{v (x)}^+ \Xi\left(\zeta (x), \zeta (x), u (t,x)\right)
      \phi (t,x) \, \d{\norma{\mu}}(x)\, \d{t}
  \end{eqnarray}
  where $v$, $\mu$ are as in~\eqref{eq:20}, and moreover
  $u (0) = u_o$.
\end{definition}

In the last integral in~\eqref{eq:48}, the integrand is Borel
measurable in $(t,x)$ since, for instance, by the above assumptions on
$u$, we have at \emph{every} $(t,x) \in \reali_+ \times \reali$
\begin{displaymath}
  u (t,x) = \lim_{h\to 0} \dfrac{1}{h}\int_{x-h}^x u (t,y) \, \d{y} \,.
\end{displaymath}
Moreover, Borel measurability on $\reali^2$ ensures measurability with
respect to the product measure.

Note that the value of the integrand in the first line
in~\eqref{eq:48} is independent of changes of the integrand on sets of
\emph{Lebesgue} measure $0$ in $\reali^2$, while the latter integrand
is integrated with respect to the product measure
$\norma{\mu} \otimes \d{t}$. Nevertheless, \eqref{eq:48} is
meaningful, since $u$ is prescribed \emph{pointwise}, at \emph{every}
point and not merely almost everywhere.

The above definition is known not to guarantee
uniqueness. Nevertheless, Theorem~\ref{thm:sgrp} below does guarantee
uniqueness, relying on an extension to the case of~\eqref{eq:21} the
precise characterization originally provided in~\cite{Bressan} for
homogeneous systems of conservation laws.

\begin{definition}
  \label{def:GenRP}
  By \emph{Generalized Riemann Problem} we mean the Cauchy
  Problem~(\ref{eq:21}) with $\zeta$ and the initial datum $u_o$ as
  follows:
  \begin{equation}
    \label{eq:25}
    \zeta (x)
    =
    z^- \, \caratt{\left]-\infty, 0\right[} (x)
    +
    z^+ \, \caratt{\left]0, +\infty\right[} (x)
    \quad \mbox{ and } \quad
    u_o (x)
    =
    u^\ell \, \caratt{\left]-\infty, 0\right[} (x)
    +
    u^r \, \caratt{\left]0, +\infty\right[} (x) \,.
  \end{equation}
  For $z \in \mathcal{Z}$ and $u \in \Omega$, call
  $\sigma_i \to H_i (z,\sigma_i) (u)$ the Lax curve of the $i$--th
  family w.r.t.~$f(z,\cdot)$ exiting $u$,
  see~\cite[\S~5.2]{Bressan2000} or~\cite[\S~9.3]{Dafermos2000}. For
  $\pmb\sigma \equiv (\sigma_1, \ldots, \sigma_n)$, we use below the
  notation
  \begin{equation}
    \label{eq:6}
    H (z,\pmb\sigma)
     =
    H_n (z,\sigma_n) \circ H_{n-1} (z,\sigma_{n-1})
    \circ \cdots \circ
    H_2 (z,\sigma_2) \circ H_1 (z, \sigma_1) \, (u) \,.
  \end{equation}
  Introduce recursively the states $w_0, \ldots, w_{n+1} \in \Omega$
  with $w_0 = u^\ell,\, w_{n+1} = u^r$ and
  \begin{equation}
    \label{eq:2}
    \left\{
      \begin{array}{l@{\qquad}r@{\,}c@{\,}l}
        w_{i+1}
        =
        H_{i+1}(z^+,\sigma_{i+1})(w_i)
        & \mbox{ if }i
        & =
        & 0, \ldots, i_o-1,
        \\
        f (z^+,w_{i_o+1}) - f (z^-,w_{i_o})
        =
        \Xi (z^+,z^-,w_{i_o})
        \\
        w_{i+1}
        =
        H_{i}(z^-,\sigma_{i})(w_i)
        & \mbox{ if }i
        & =
        &i_o+1, \ldots, n\,.
      \end{array}
    \right.
  \end{equation}
  We thus define as \emph{Admissible Solution} to the Generalized
  Riemann Problem~\eqref{eq:21}--\eqref{eq:25} the gluing along
  $x = 0$ of the Lax solutions to the (standard) Riemann Problems
  \begin{displaymath}
    \left\{
      \begin{array}{l}
        \partial_t u + \partial_x f (z^-,u) = 0
        \\
        u (0,x)
        = u^\ell \caratt{\left]-\infty, 0\right[} (x)
        +
        w_{i_o} \caratt{\left]0, +\infty\right[} (x),
      \end{array}
    \right.
    \quad
    \left\{
      \begin{array}{l@{}}
        \partial_t u + \partial_x f (z^+,u) = 0
        \\
        u (0,x)
        = w_{i_o+1} \caratt{\left]-\infty, 0\right[} (x)
        +
        u^r \caratt{\left]0, +\infty\right[} (x).
      \end{array}
    \right.
  \end{displaymath}
\end{definition}
\noindent Throughout, we refer to the stationary jump discontinuities
due to jumps in $z$ as to \emph{zero waves}.

Below, Lemma~\ref{lem:RP} ensures that, with the above definition, the
Generalized Riemann Problem~(\ref{eq:21})--\eqref{eq:25} turns out to
be well posed.

\bigskip

Aiming at the characterization of solutions to~\eqref{eq:21}, we now
extend to the present case the general definitions introduced
in~\cite{Bressan}, see also~\cite[Chapter~9]{Bressan2000}. Fix
$\zeta \in \BV (\reali; \mathcal{Z})$ and a function $u = u (t,x)$
with $u (t) \in \BV (\reali; \Omega)$ for all $t$ and a point
$(\tau,\xi)\in \mathopen[0,+\infty\mathclose[\times\reali$. Define the
function $U^\sharp_{(u;\tau,\xi)}$ as the solution to the generalized
Riemann Problem
\begin{equation}
  \label{eq:64}
  \left\{
    \begin{array}{l}
      \partial_t U + \partial_x f\left(\zeta (\xi), U\right) =
      \Xi\left(\zeta (\xi+), \zeta (\xi), u (t,\xi-)\right) \,
      \delta_{\xi}
      \\
      U (0,x)
      =
      \left\{
      \begin{array}{ll}
        u (\tau, \xi-)
        & x < \xi \,;
        \\
        u (\tau, \xi+)
        & x > \xi \,.
      \end{array}
          \right.
    \end{array}
  \right.
\end{equation}
Note that if $\xi \not\in\mathcal{I} (\zeta)$, then the right hand
side in~\eqref{eq:64} vanishes due to~\ref{eq:Xi2} and the above
definition of $U^\sharp_{(u;\tau,\xi)}$ reduces to the classical one
in~\cite[Chapter~9]{Bressan} related to the homogeneous flow
$u \to f\left(\zeta (\xi),u\right)$.

We define the function $U^\flat_{(u;\tau,\xi)}$ as the solution to the
following linear hyperbolic problem with constant coefficients and
measure-valued source term
\begin{equation}
  \label{eq:HelpC}
  \left\{
    \begin{array}{l}
      \partial_t U + A\; \partial_x U = g
      \\
      U(0,x) = u(\tau,x)
    \end{array}
  \right.
\end{equation}
with $A = D_u f\left(\zeta(\xi),u(\tau,\xi)\right)$ and for any Borel
subset $E$ of $\reali$,
\begin{equation}
  \label{eq:57}
  \!\!\!
  \begin{array}{@{}rcl@{}}
    g(E)
    &  =
    & \displaystyle
      \sum_{{\bar x} \in \mathcal{I} (\zeta)}
      \left(
      \Xi\left(\zeta ({\bar x}+), \zeta ({\bar x}), u (\tau, \xi)\right)-f\left(\zeta ({\bar x}+),u (\tau, \xi)\right)+f\left(\zeta ({\bar x}),u (\tau, \xi)\right)\right)
      \delta_{{\bar x}}(E)
    \\
    &
    &  \displaystyle
      \; +
      \int_E
      \left(D_{v (x)}^+ \Xi\left(\zeta (x), \zeta (x), u (\tau,\xi)\right)
      - D_z f\left(\zeta(x),u(\tau,\xi)\right) \, v (x)\right) \d{\norma{\mu} (x)}.
  \end{array}
  \!\!
\end{equation}
where we used the same notation as in~\eqref{eq:20} and~\eqref{eq:26}.

We are now ready to state the main result of this work.

\begin{theorem}
  \label{thm:sgrp}
  Let $f$ satisfy~\ref{it:f1}--\ref{it:f4}, $\Xi$
  satisfy~\ref{eq:Xi1}--\ref{eq:Xi4}. Fix $\bar z \in \mathcal{Z}$,
  $\bar u \in \Omega$. Then, there exist positive $\delta$ and $L$
  such that for any $\zeta \in \BV (\reali; \mathcal{Z})$ with
  $\tv (\zeta) < \delta$ and $\norma{\zeta (x) - \bar z} < \delta$
  there exists a domain
  $\mathcal{D}^\zeta \subseteq \bar u + \L1 (\reali; \Omega)$
  containing all functions $u$ in $\bar u + \L1 (\reali; \Omega)$ with
  $\tv (u) < \delta$ and a semigroup
  $S^\zeta \colon \reali_+ \times \mathcal{D}^\zeta \to
  \mathcal{D}^\zeta$ such that
  \begin{enumerate}
  \item For all $u_o \in \mathcal{D}^\zeta$, the orbit
    $t \to S^\zeta_t u_o$ solves \eqref{eq:21} in the sense of
    Definition~\ref{def:sol}.

  \item $S^\zeta$ is $\L1$--Lipschitz continuous, i.e.~for all
    $u_o, u_o^1, u_o^2 \in \mathcal{D}^\zeta$ and for all
    $t,t_1,t_2 \in \reali_+$
    \begin{eqnarray*}
      \norma{S^\zeta_t u_o^1 - S^\zeta_t u_o^2}_{\L1 (\reali; \reali^n)}
      &\leq
      & L \; \norma{u_o^1 - u_o^2}_{\L1 (\reali; \reali^n)} \,;
      \\
      \norma{S^\zeta_{t_1} u_o - S^\zeta_{t_2} u_o}_{\L1 (\reali; \reali^n)}
      &\leq
      & L \; \modulo{t_1 - t_2} \,.
    \end{eqnarray*}

  \item If $\zeta \in \PC (\reali; \mathcal{Z})$ and
    $u_o \in \PC (\reali; \Omega)$, then for $t$ sufficiently small,
    the map $(t,x) \to (S^\zeta_t u_o) (x)$ coincides with the gluing
    of Admissible Solutions, in the sense of
    Definition~\eqref{def:GenRP}, to Generalized Riemann Problems at
    the points of jumps of $u_o$ and of $\zeta$.

    \end{enumerate}
    Moreover, let $\hat\lambda$ be an upper bound for the (moduli of)
    characteristic speeds and define $u (t,x) = (S^\zeta_t u_o)
    (x)$. Then, for every $(\tau, \xi) \in \reali_+ \times \reali$,
    \begin{enumerate}[label=(\roman*)]
    \item
      \begin{displaymath}
        \lim_{\theta\to 0}
        \frac{1}{\theta}
        \int_{\xi-\theta\hat\lambda}^{\xi+\theta\hat\lambda}
        \modulo{u(\tau+\theta,x) - U^\sharp_{(u;\tau,\xi)}(\theta,x)} \d x
        =
        0 \,.
      \end{displaymath}
    \item There exists a constant $C$ such that for every
      $a,b \in \reali$ with $a < \xi < b$ and for every
      $\theta \in \mathopen]0, (b-a)/(2\hat\lambda)\mathclose[$,
      \begin{displaymath}
        \frac{1}{\theta}
        \int_{a+\theta\hat\lambda}^{b-\theta\hat\lambda}
        \modulo{u(\tau+\theta,x)-U^\flat_{(u;\tau,\xi)}(\theta,x)} \d x
        \leq
        C\left[\tv\left(u(\tau), \mathopen]a,b\mathclose[\right)
          +
          \tv\left(\zeta, \mathopen]a,b\mathclose[\right)
        \right]^2.
      \end{displaymath}
    \end{enumerate}
    If $u\colon [0,T]\to \mathcal{D^\zeta}$ is $\L1$--Lipschitz
    continuous and satisfies~(i) and~(ii) for almost every time $\tau$
    and for all $\xi \in \reali$, then $t \to u(t,\cdot)$ coincides
    with an orbit of the semigroup $S^\zeta$.
  \end{theorem}

  Note that whenever $\zeta$ is piecewise constant, the properties~1.,
  2.~and~3.~above uniquely characterize the semigroup $S^\zeta$, see
  Lemma~\ref{lem:123}.

  \section{Proofs}
  \label{sec:assumpt-main-results}

  Below, $\O$ denotes a constant depending exclusively on $f$, $\Xi$
  and on a neighborhood of $\bar u$. By $\hat\lambda$ we denote an
  upper bound for (the moduli) of characteristic speeds.

  \subsection{Preliminary Results}
  \label{sub:GerneralSingle}

  First, we recall a Lipschitz-type estimate on the map $\Xi$, of use
  throughout this paper.

  \begin{lemma}[{\cite[Lemma~4.3]{ColomboGuerraHolle}}]
    \label{lem:EstimateXi2}
    Let $W \subset \reali^m$ be non empty, open, bounded and
    convex. Let
    $\phi \colon \mathcal{Z} \times \mathcal{Z} \to \C1 (\overline{W};
    \reali^n)$ be Lipschitz continuous and such that $\phi (z,z,w)=0$
    for every $z \in \mathcal{Z}$ and $w \in W$. Then,
    \begin{equation}
      \label{eq:lipXi}
      \begin{array}{rcl}
        \norma{\phi (z^+,z^-,w)}
        & \leq
        & \O \, \norma{z^+ - z^-}
        \\
        \norma{\phi (z^{+},z^{-},w_{2}) - \phi (z^+,z^{-},w_{1})}
        & \leq
        & \O \, \norma{z^{+}-z^{-}} \, \norma{w_{2}-w_{1}} \, .
      \end{array}
    \end{equation}
  \end{lemma}

\begin{proof}
  Since $\phi (z^-,z^-,w)=0$, we have
  $\norma{\phi (z^+,z^-,w)} = \norma{\phi (z^+,z^-,w)- \phi
    (z^-,z^-,w)}$ and the first inequality in~\eqref{eq:lipXi} follows
  by the global Lipschitz continuity of $\phi$ with respect to the $z$
  variables.

  Observe that $D_w \phi (z^-,z^-,w)=0$. Hence, using again the
  Lipschitz continuity of $\phi$,
  \begin{eqnarray*}
    &
    &
      \norma{\phi (z^{+},z^{-},w_{2}) - \phi (z^+,z^{-},w_{1})}
    \\
    & =
    &\norma{
      \int_0^1 D_w \phi \left(z^+,z^-,w_2+\varsigma (w_1-w_2)\right) (w_1-w_2)
      \d\varsigma
      }
    \\
    & =
    & \norma{
      \int_0^1 \!
      \left[
      D_w \phi \left(z^+,z^-,w_2+\varsigma (w_1-w_2)\right)
      {-}
      D_w \phi \left(z^-,z^-,w_2+\varsigma (w_1-w_2)\right)
      \right]
      (w_1-w_2)
      \d\varsigma
      }
    \\
    & \leq
    & \O \, \norma{z^{+}-z^{-}} \, \norma{w_{2}-w_{1}} \, .
  \end{eqnarray*}
\end{proof}

Note that~\ref{eq:Xi1} and~\ref{eq:Xi3} are stronger than the
assumptions in Lemma~\ref{lem:EstimateXi2}, so that $\Xi$
satisfies~\eqref{eq:lipXi}.

Introduce a map $T$ related to the generalized Riemann Problem.

\begin{lemma}
  \label{lem:ExistenceT}
  Let $f$ satisfy~\ref{it:f1}--\ref{it:f4} and $\Xi$
  satisfy~\ref{eq:Xi1}, \ref{eq:Xi3}. Then, for any
  $\bar z \in \mathcal{Z}$ and $\bar u \in \Omega$, there exists
  $\delta > 0$ and a Lipschitz map
  $T \colon B(\bar z;\delta)^2 \to \C2\left( B(\bar u; \delta);
    \Omega\right)$ such that
  \begin{equation}
    \label{eq:1}
    \begin{cases}
      f (z^+,u^+) - f (z^-,u^-) = \Xi (z^+,z^-,u^-)
      \\
      z^+, z^-\in B(\bar z;\delta)
      \\
      u^+,u^-\in B(\bar u;\delta)
    \end{cases}
    \Longleftrightarrow \quad u^+=T(z^+, z^-)(u^-) \, .
  \end{equation}
  Furthermore,
  \begin{enumerate}
  \item $T (z,z) (u) = u$ and the map
    $(z^+,z^-,u) \to T (z^+,z^-) (u)-u$ satisfies the assumptions of
    Lemma~\ref{lem:EstimateXi2}.
  \item The following expansion holds:
    \begin{eqnarray*}
      &
      & f (z^+,u^*)- f (z^-, u^*) - \Xi (z^+,z^-,u^*)
        + D_u f (z^*,u^*)
        \left(T (z^+,z^-) (u) - u\right)
      \\
      & =
      & \O \, \norma{z^+-z^-}
        \left(
        \norma{z^+-z^*} + \norma{z^+-z^-}+\norma{u-u^*}
        \right)
    \end{eqnarray*}
  \end{enumerate}
\end{lemma}

\begin{proof}
  (This Lemma is an extension of~\cite[Lemma~4.4]{ColomboGuerraHolle}
  to the case $f$ dependent on $z$, too.)

  Since $\bar u \in \Omega$, \ref{it:f1} and~\ref{it:f2} ensure that
  the function $u\to f(z,u)$ has a local $\C2$ inverse $\phi$, in the
  sense that $\phi\left(z, f (z,u)\right)=u$, for $z$ sufficiently
  close to $\bar z$. Define
  \begin{equation}
    \label{eq:defT}
    T (z^+, z^-)(u^{-})
     =
    \phi\left(z^+,f(z^{-},u^{-}) + \Xi(z^{+},z^{-},u^{-})\right).
  \end{equation}
  $T$ enjoys the required Lipschitz regularity and moreover
  \begin{displaymath}
    T (z,z) (u)
    =
    \phi\left(z,f(z,u) + \Xi(z,z,u)\right)
    =
    \phi\left(z,f(z,u)\right)
    =
    u \,.
  \end{displaymath}
  To prove~2., rewrite
  \begin{equation}
    \label{eq:49}
    f (z^+,u^*)- f (z^-, u^*) - \Xi (z^+,z^-,u^*)
    + D_u f (z^*,u^*)
    \left(T (z^+,z^-) (u) - u\right)
    =
    \mathcal{E}_1+\mathcal{E}_2+\mathcal{E}_3
  \end{equation}
  where we used the definition of $T$ and set
  \begin{eqnarray*}
    \mathcal{E}_1
    & =
    & f (z^+,u^*) - f\left(z^+, T (z^+,z^-) (u^*)\right)
      + D_u f (z^+,u^*)\left(T (z^+,z^-) (u^*)-u^*\right)
    \\
    \mathcal{E}_2
    & =
    & -D_u f (z^+,u^*)\left(T (z^+,z^-) (u^*)-u^*\right)
      +D_u f (z^*,u^*)\left(T (z^+,z^-) (u^*)-u^*\right)
    \\
    \mathcal{E}_3
    & =
    & - D_u f (z^*,u^*) \left(T (z^+,z^-) (u^*)-u^*\right)
      + D_u f (z^*,u^*) \left(T (z^+,z^-) (u) - u\right)
  \end{eqnarray*}
  By a Taylor expansion, we have:
  \begin{eqnarray*}
    \norma{\mathcal{E}_1}
    & \leq
    & \O \, \norma{T (z^+,z^-)(u^*) - u^*}^2
    \\
    & =
    & \O \, \norma{z^+ - z^-}^2 \,.
  \end{eqnarray*}
  Concerning $\mathcal{E}_2$,
  \begin{eqnarray*}
    \norma{\mathcal{E}_2}
    & =
    & \norma{D_u f (z^+,u^*) - D_u f (z^*,u^*)} \;
      \norma{T (z^+,z^-) (u^*)-u^*}
    \\
    & \leq
    & \O \, \norma{z^+-z^*} \; \norma{z^+-z^-} \,.
  \end{eqnarray*}
  Finally, by Lemma~\ref{lem:EstimateXi2}
  \begin{eqnarray*}
    \norma{\mathcal{E}_3}
    & =
    & \norma{D_u f (z^*,u^*)} \;
      \norma{
      \left(T (z^+,z^-) (u^*)-u^*\right)
      -
      \left(T (z^+,z^-) (u) - u\right)
      }
    \\
    & \leq
    & \O \, \norma{z^+ - z^-} \; \norma{u-u^*} \,,
  \end{eqnarray*}
  completing the proof.
\end{proof}

As a consequence, we also prove the well posedness of the Generalized
Riemann Problem~(\ref{eq:21})--\eqref{eq:25}.

\begin{lemma}
  \label{lem:RP}
  Let $f$ satisfy~\ref{it:f1}--\ref{it:f4} and $\Xi$
  satisfy~\ref{eq:Xi1}. Then, there exists a positive $\delta$ such
  that if $u_\ell,u_r \in \Omega$ and $z^+, z^-\in \mathcal Z$ satisfy
  \begin{displaymath}
    \norma{u_\ell-u_r} \leq \delta \,,\qquad
    \norma{z^+ - z^-} \leq \delta \,,
  \end{displaymath}
  then, the Generalized Riemann Problem~\eqref{eq:21}--\eqref{eq:25}
  admits a unique solution in the sense of
  Definition~\ref{def:GenRP}. Moreover, the waves' sizes
  $(\sigma_1, \ldots, \sigma_n)$ and the states $(w_1, \ldots, w_n)$
  in~\eqref{eq:2} exist, are uniquely defined and are Lipschitz
  continuous functions of $z^+, z^-, u_r, u_\ell$.
\end{lemma}

\begin{proof}
  Simply rewrite~\eqref{eq:2} by means of~\eqref{eq:1} to
  use~\cite[Lemma~3]{AmadoriGosseGuerra2002}.
\end{proof}

The following notation is of use below:
\begin{equation}
  \label{eq:39}
  (\sigma_1, \ldots, \sigma_n)
  =
  E(z^+, z^-, u_r,  u_\ell) \,.
\end{equation}
We separate the waves with negative ($\pmb\sigma'$) or positive
($\pmb\sigma''$) propagation speed as follows:
\begin{align}
  \begin{split}
    \pmb\sigma' =  (\sigma_1,\dots,\sigma_{i_o},0,\dots,0),
    &\qquad \pmb\sigma''  =
    (0,\dots,0,\sigma_{i_o+1},\dots,\sigma_n),
    \\
    \pmb\sigma &  =  \pmb\sigma'+\pmb\sigma''\in\reali^n.
  \end{split}
\end{align}
Given two $n$-tuples of waves $\pmb{\alpha}$ and $\pmb{\beta}$, the
waves $i$ with size $\alpha_i \neq 0$ and $j$ with size
$\beta_j \neq 0$ are \emph{approaching} whenever $i>j$ or $i=j$, the
$i$--th family is genuinely nonlinear and
$\min\left\{\alpha_i, \beta_j\right\} < 0$,
see~\cite[\S~7.3]{Bressan2000} or~\cite[\S~9.9]{Dafermos2000}. Call
$\mathcal{A}_{\pmb{\alpha},\pmb{\beta}}$ the set of these pairs
$(i,j)$.

\begin{lemma}[{\cite[Theorem p.~30]{Yong1999}}]
  \label{lem:app}
  Let
  $\phi \in \C{2,1} (\overline{B (0, \bar\delta)}\times\overline{B (0,
    \bar\delta)}; \reali^m)$ be such that
  \begin{equation}
    \label{eq:16}
    \phi (\pmb\alpha,\pmb\beta)=0
    \quad \mbox{ for all } \quad
    \pmb\alpha, \pmb\beta
    \quad \mbox{ with } \quad
    \mathcal{A}_{\pmb\alpha,\pmb\beta}=\emptyset \,.
  \end{equation}
  Then, for all $\pmb\alpha,\pmb\beta$
  \begin{displaymath}
    \norma{\phi (\pmb\alpha;\pmb\beta)}
    \leq
    \O
    \sum_{(i,j) \colon i>j}
    \modulo{\alpha_i \, \beta_j}
    +
    \O
    \left(\norma{\pmb\alpha} + \norma{\pmb\beta}\right)
    \sum_{i \colon {\min \{\alpha_i,\beta_i\}<0 \atop \mbox{\tiny gen. nonl.}}}
    \modulo{\alpha_i \, \beta_i}
    \,.
  \end{displaymath}
\end{lemma}

\begin{proof}
  Observe that for all $\pmb\alpha, \pmb\beta$ in
  $\overline{B (0, \bar\delta)}$, we have
  $\mathcal{A}_{\pmb\alpha,0} = \mathcal{A}_{0,\pmb\beta} =
  \emptyset$. Hence,
  \begin{displaymath}
    \phi (\pmb\alpha;0) = \phi (0;\pmb\beta) = 0
    \quad \mbox{ and } \quad
    \partial_{\alpha_i} \phi (\pmb\alpha;0) = \partial_{\beta_j} \phi (0;\pmb\beta) = 0
  \end{displaymath}
  for all $i,j=1, \ldots, n$. Following~\cite{Yong1999}, we have
  \begin{eqnarray*}
    &
    & \norma{\phi (\pmb\alpha;\pmb\beta)}
    \\
    & =
    & \norma{\phi (\pmb\alpha;\pmb\beta)- \phi (\pmb\alpha; 0)}
    \\
    & \leq
    & \sum_{i=1}^n
      \norma{
      \phi (\alpha_1, \ldots, \alpha_i, 0, \ldots,0;\pmb\beta)
      -
      \phi (\alpha_1, \ldots, \alpha_{i-1}, 0, \ldots,0;\pmb\beta)
      }
    \\
    & \leq
    & \sum_{i=1}^n
      \int_0^{\alpha_i}
      \norma{
      \partial_{\alpha_i} \phi (\alpha_1, \ldots, \alpha_{i-1},a,0, \ldots,0;\pmb\beta)
      }
      \d{a}
    \\
    & =
    & \sum_{i=1}^n
      \int_0^{\alpha_i}
      \norma{
      \partial_{\alpha_i} \phi (\alpha_1, \ldots, \alpha_{i-1},a,0, \ldots,0;\pmb\beta)
      -
      \partial_{\alpha_i} \phi (\alpha_1, \ldots, \alpha_{i-1},a,0, \ldots,0;0)
      }
      \d{a}
    \\
    & \leq
    & \sum_{i=1}^n
      \sum_{j=1}^n
      \int_0^{\alpha_i}
      \left\|
      \partial_{\alpha_i} \phi (\alpha_1, \ldots, \alpha_{i-1},a,0, \ldots,0;0, \ldots, 0, \beta_j, \ldots, \beta_n)
      \right.
    \\
    &
    & \qquad\qquad\qquad
      \left.
      -
      \partial_{\alpha_i} \phi (\alpha_1, \ldots, \alpha_{i-1},a,0, \ldots,0;0, \ldots, 0, \beta_{j+1}, \ldots, \beta_n)
      \right\|
      \d{a}
    \\
    & \leq
    & \sum_{i=1}^n
      \sum_{j=1}^n
      \int_0^{\alpha_i} \int_0^{\beta_j}
      \norma{
      \partial_{\alpha_i} \partial_{\beta_j}
      \phi (\alpha_1, \ldots, \alpha_{i-1},a,0, \ldots,0;0, \ldots, 0, b, \beta_{j+1}, \ldots, \beta_n)
      }
      \d{b} \d{a}
    \\
    & \leq
    & \sum_{(i,j) \in \mathcal{A}_{\pmb\alpha,\pmb\beta}}
      \int_0^{\alpha_i} \int_0^{\beta_j}
      \norma{
      \partial_{\alpha_i} \partial_{\beta_j}
      \phi (\alpha_1, \ldots, \alpha_{i-1},a,0, \ldots,0;0, \ldots, 0, b, \beta_{j+1}, \ldots, \beta_n)
      }
      \d{b} \d{a}
    \\
    & \leq
    & \norma{D^2\phi}_{\C0} \sum_{i>j} \modulo{\alpha_i \, \beta_j}
      +
      \Lip(D^2\phi)
      \left(\norma{\pmb\alpha} + \norma{\pmb\beta}\right)
      \sum_{i \colon {\min \{\alpha_i,\beta_i\}<0 \atop \mbox{\tiny gen. nonl.}}} \modulo{\alpha_i \, \beta_j} \,.
  \end{eqnarray*}
  Above, we noted that some terms in the latter double sum vanish
  by~\eqref{eq:16}, since
  \begin{eqnarray*}
    (i,j) \not\in \mathcal{A}_{\pmb\alpha,\pmb\beta}
    & \implies
    & \mathcal{A}_{\alpha_1, \ldots, \alpha_{i-1},a,0, \ldots,0; 0, \ldots, 0,b,\beta_{j+1}, \ldots, \beta_n} = \emptyset
      \quad \mbox{ for all }
      \begin{array}{c}
        a \mbox{ between } 0 \mbox{ and } \alpha_i;
        \\
        b \mbox{ between } 0 \mbox{ and } \beta_j.
      \end{array}
  \end{eqnarray*}
  In the terms with $i>j$, we use a standard estimate bounding the
  integral by means of the $\C0$ norm. We are left with the terms with
  $i=j$, the $i$--th field is genuinely nonlinear and
  $\min\{\alpha_i,\beta_i\}<0$. In this case, \eqref{eq:16} ensures
  that
  \begin{displaymath}
    \phi (\alpha_1, \ldots, \alpha_{i-1}, a, 0, \ldots, 0;0, \ldots, 0, b, \beta_{i+1}, \ldots, \beta_n) = 0
    \quad \mbox{ for all }a\geq 0 \mbox{ and } b \geq 0\,.
  \end{displaymath}
  Hence $\partial_{\alpha_1} \partial_{\beta_i} \phi (0;0) = 0$ and
  $\norma{\partial_{\alpha_i}\partial_{\beta_i} \phi
    (\pmb\alpha,\pmb\beta)} \leq \Lip (D^2\phi)
  \left(\norma{\pmb\alpha}+\norma{\pmb\beta}\right)$.
\end{proof}

\begin{lemma}[{\cite[Lemma~4]{AmadoriGosseGuerra2002}
    and~\cite[Lemma~4.8]{ColomboGuerraHolle}}]
  \label{lem:AGG4}
  Let $f$ satisfy~\ref{it:f1}--\ref{it:f4}, $\Xi$ satisfy~\ref{eq:Xi1}
  and~\ref{eq:Xi3}.  Then, there exists a positive $\delta$ such that
  if $u_\ell,u_r \in \Omega$ and $z^+, z^-\in \mathcal Z$ are such
  that
  \begin{displaymath}
    \norma{u_\ell-u_r} \leq \delta \,,\qquad
    \norma{z^+ - z^-} \leq \delta
  \end{displaymath}
  and if $\pmb\sigma = E (z^+, z^-, u_r, u_\ell)$ is as
  in~\eqref{eq:39}, we have
  \begin{displaymath}
    \norma{u_r - u_\ell}
    = \O \left(\norma{\pmb\sigma} + \norma{z^+-z^-}\right)
    \quad \mbox{ and } \quad
    \norma{\pmb\sigma}
    = \O \left(\norma{u_r -  u_\ell} + \norma{z^+-z^-} \right)\,.
  \end{displaymath}
\end{lemma}

\begin{lemma}
  \label{lem:milan}
  Let $f$ satisfy~\ref{it:f1}--\ref{it:f4}. For all
  $z \in \mathcal{Z}$, $u \in \Omega$ and for all sufficiently small
  $\pmb\alpha, \pmb\beta \in \reali^n$
  \begin{eqnarray}
    \label{eq:7}
    \norma{
    H (z,\pmb\beta) \circ H (z,\pmb\alpha) (u)
    -
    H (z, \pmb\alpha+\pmb\beta) (u)
    }
    & \leq
    & \O \, \sum_{(\alpha_i,\beta_i) \in \mathcal{A}_{\pmb\alpha,\pmb\beta}}
      \modulo{\alpha_i \, \beta_i}
    \\
    \label{eq:8}
    \norma{H (z^+,\pmb\alpha) (u) - H (z^-,\pmb\alpha) (u)}
    & \leq
    & \O \, \norma{z^+ - z^-} \sum_{i=1}^n \modulo{\alpha_i}
  \end{eqnarray}
\end{lemma}

\begin{proof}
  The classical Glimm interaction estimate~\eqref{eq:7} follows from
  Lemma~\ref{lem:app} with
  $f (\pmb\alpha,\pmb\beta) = H (z,\pmb\beta) \circ H (z,\pmb\alpha)
  (u) - H (z, \pmb\alpha+\pmb\beta) (u)$.

  To obtain the second, apply Lemma~\ref{lem:EstimateXi2} with
  $w_2 = \pmb\alpha$, $w_1 = 0$ and
  $\phi (z^+,z^-,\pmb\alpha) = H (z^+,\pmb\alpha) (u) - H
  (z^-,\pmb\alpha) (u)$.
\end{proof}

The following lemma comprises the interaction estimates necessary
below.

\begin{lemma}
  \label{lem:inter}
  Let $f$ satisfy~\ref{it:f1}--\ref{it:f4} and $\Xi$
  satisfy~\ref{eq:Xi1}, \ref{eq:Xi3}.  Then, there exists a positive
  $\delta$ such that if $u_\ell,u_r \in \Omega$;
  $z^+, z^-\in \mathcal Z$ and $\pmb\alpha, \pmb\beta \in \reali^n$
  are such that
  \begin{displaymath}
    \norma{u_\ell-u_r} \leq \delta \,,\qquad
    \norma{z^+ - z^-} \leq \delta \,,\qquad
    \norma{\pmb\alpha} + \norma{\pmb\beta} \leq \delta
  \end{displaymath}
  with reference to Figure~\ref{fig:gamma}, the following general
  interaction estimate holds:
  \begin{equation}
    \label{eq:9}
    \norma{u_*-u_r}
    \leq
    \O
    \left(
      \sum_{(i,j)\in\mathcal{A}_{\pmb\alpha,\pmb\beta}}
      \modulo{\alpha_i \, \beta_j}
      +
      \norma{z^+ - z^-}
      \sum_{i>i_o} \modulo{\alpha_i}
    \right) \,.
  \end{equation}
\end{lemma}

\begin{figure}[H]
  \centering
  \begin{tikzpicture}[line cap=round,line join=round,x=0.8cm,y=0.8cm]
    \draw [line width=0.5pt] (-8,0) -- (8,0); \draw [line width=0.5pt,
    line width=1pt] (-4,0) -- (-7,4); \draw [line width=0.5pt, line
    width=1pt] (-4,0) -- (-1,4); \draw [line width=0.5pt, line
    width=1pt] (4,0) -- (7,4); \draw [line width=0.5pt, line
    width=1pt] (4,0) -- (1,4); \draw [line width=0.5pt, line
    width=1pt] (4,0) -- (4,4); \draw (-1,4) node
    [label=above:{$\alpha''$}] {}; \draw (-7,4) node
    [label=above:{$\alpha'$}] {}; \draw (1,4) node
    [label=above:{$\beta'$}] {}; \draw (7,4) node
    [label=above:{$\beta''$}] {}; \draw (4,4) node
    [label=above:{$\Delta z$}] {}; \draw (-7,0) node
    [label=above:{$u=u_{\ell}$}] {}; \draw (3,3) node
    [label=above:{$z=z^-$}] {};\draw (5,3) node
    [label=above:{$z=z^+$}] {}; \draw (7,0) node
    [label=above:{$u=u_{r}$}] {};
  \end{tikzpicture}
  \begin{tikzpicture}[line cap=round,line join=round,x=0.8cm,y=0.8cm]
    \draw [line width=0.5pt] (-8,0) -- (8,0); \draw [line width=0.5pt,
    line width=1pt] (0,0) -- (3,4); \draw [line width=0.5pt, line
    width=1pt] (0,0) -- (-3,4); \draw [line width=0.5pt, line
    width=1pt] (0,0) -- (0,4); \draw [line width=0.5pt, line
    width=1pt] (0,0) -- (7,4); \draw (-3,4) node
    [label=above:{$\alpha'+\beta'$}] {}; \draw (3,4) node
    [label=above:{$\alpha''+\beta''$}] {}; \draw (0,4) node
    [label=above:{$\Delta z$}] {}; \draw (7,4) node
    [label=above:{$\gamma$}] {}; \draw (7,5) node [label=above:{$ $}]
    {}; \draw (-6,0) node [label=above:{$u=u_{\ell}$}] {}; \draw (6,0)
    node [label=above:{$u=u_{r}$}] {}; \draw (4,3) node
    [label=above:{$u=u_{*}$}] {}; \draw (-1.5,3) node
    [label=above:{$z=z^-$}] {};\draw (1.5,3) node
    [label=above:{$z=z^+$}] {};
  \end{tikzpicture}
  \caption{\small Notation used in Lemma~\ref{lem:inter}. $\gamma$
    denotes a fictitious wave separating the states $u_*$, as defined
    in~\eqref{eq:11}, and $u_r$. $\Delta z$ denotes the zero wave
    between $z^-$ and $z^+$.}
  \label{fig:gamma}
\end{figure}
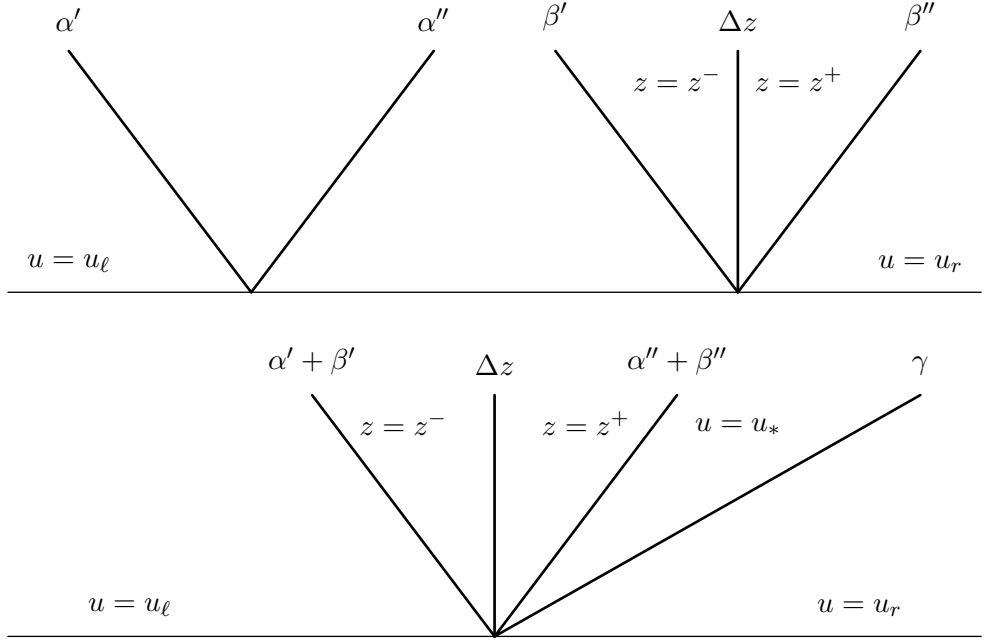

\begin{proof}
  Referring to Figure~\ref{fig:gamma}, we have:
  \begin{eqnarray}
    \nonumber
    u_r
    & =
    & H(z^+,\pmb\beta'')
      \circ
      T (z^+,z^-)
      \circ
      H (z^-,\pmb\beta')
      \circ
      H (z^-, \pmb\alpha'')
      \circ
      H (z^-,\pmb\alpha')
      (u_\ell)
    \\
    \label{eq:11}
    u_*
    & =
    & H (z^+,\pmb\alpha'' + \pmb\beta'')
      \circ
      T (z^+,z^-)
      \circ
      H (z^-,\pmb\alpha'+\pmb\beta') (u_\ell) \,.
  \end{eqnarray}
  Introduce
  \begin{eqnarray*}
    \check u
    & =
    & H (z^+,\pmb\beta'')
      \circ
      T (z^+,z^-)
      \circ
      H (z^-,\pmb\alpha'')
      \circ
      H (z^-,\pmb\beta')
      \circ
      H (z^-,\pmb\alpha') (u_\ell)
    \\
    \hat u
    & =
    & H (z^+,\pmb\beta'')
      \circ
      H (z^+,\pmb\alpha'')
      \circ
      T (z^+,z^-)
      \circ
      H (z^-,\pmb\beta')
      \circ
      H (z^-,\pmb\alpha') (u_\ell)
  \end{eqnarray*}
  so that
  \begin{displaymath}
    \norma{u_r - u_*}
    \leq
    \norma{u_r - \check u}
    +
    \norma{\check u - \hat u}
    +
    \norma{\hat u - u_*}
  \end{displaymath}
  and by Lemma~\ref{lem:inter}, setting
  $\tilde u = H (z^-,\pmb\alpha') (u_\ell)$,
  \begin{eqnarray*}
    \norma{u_r - \check u}
    & \leq
    & \O \,
      \norma{
      H (z^-,\beta') \circ H (z^-, \alpha'') (u)
      -
      H (z^-,\alpha'') \circ H(z^-,\beta') (u)}
    \\
    & =
    &
      \O \,
      \norma{
      H (z^-,\beta') \circ H (z^-, \alpha'') (u)
      -
      H (z^-,\alpha''+\beta') (u)}
    \\
    & \leq
    & \O \sum_{(i,j)\in \mathcal{A}_{\pmb\alpha,\pmb\beta}}
      \modulo{\alpha_i \, \beta_j} \,.
  \end{eqnarray*}
  Similarly, setting now
  $\tilde u = H (z^-, \beta') \circ H (z^-,\pmb\alpha') (u_\ell)$,
  \begin{displaymath}
    \norma{\check u - \hat u}
    \leq
    \O
    \norma{
      T (z^+,z^-) \circ H (z^-,\alpha'') (\tilde u)
      -
      H (z^-,\alpha'') \circ T (z^+,z^-) (\tilde u)
    }
  \end{displaymath}
  apply Lemma~\ref{lem:EstimateXi2} with $w_2 = \alpha''$, $w_1 = 0$
  and
  $\phi (z^+,z^-,\alpha'') = T (z^+,z^-) \circ H (z^-,\alpha'')
  (\tilde u) - H (z^-,\alpha'') \circ T (z^+,z^-) (\tilde u)$ to
  obtain
  \begin{equation}
    \label{eq:10}
    \norma{\check u - \hat u}
    \leq
    \O \,
    \norma{z^+-z^-} \, \sum_{i>i_o} \modulo{\alpha_i} \,.
  \end{equation}
  Finally, using~\eqref{eq:7} in Lemma~\ref{lem:milan},
  \begin{displaymath}
    \norma{\hat u - u_*}
    \leq
    \O \, \sum_{(i,j) \in \mathcal{A}_{\pmb\alpha,\pmb\beta}}
    \modulo{\alpha_i \, \beta_j} \,,
  \end{displaymath}
  completing the proof.
\end{proof}

Note that entirely similar estimates apply to the case where the
$\pmb\alpha$ waves are on the \emph{right} of the zero wave, i.e., in
the region where $z$ attains the value $z^+$.

\begin{lemma}
  \label{lem:superRP}
  Let $f$ satisfy~\ref{it:f1}--\ref{it:f4} and $\Xi$
  satisfy~\ref{eq:Xi1}, \ref{eq:Xi3}. Then, there exists a positive
  $\delta$ such that if $u_\ell,u_r \in \Omega$;
  $z^+, z^-\in \mathcal Z$ and $\pmb\alpha, \pmb\beta \in \reali^n$
  are such that
  \begin{displaymath}
    \norma{u_\ell-u_r} \leq \delta \,,\qquad
    \norma{z^+ - z^-} \leq \delta \,,\qquad
    \norma{\pmb\alpha} + \norma{\pmb\beta} \leq \delta
  \end{displaymath}
  with reference to Figure~\ref{fig:bo}, the following general
  interaction estimate holds:
  \begin{displaymath}
    \norma{\pmb\sigma - (\pmb\alpha+\pmb\beta)}
    \leq
    \O \left(
      \sum_{(i,j) \in \mathcal{A}_{\pmb\alpha,\pmb\beta}}
      \modulo{\alpha_i \, \beta_j}
      +
      \norma{z^+ - z^-} \sum_{i>i_o} \modulo{\alpha_i}
    \right)
  \end{displaymath}
\end{lemma}

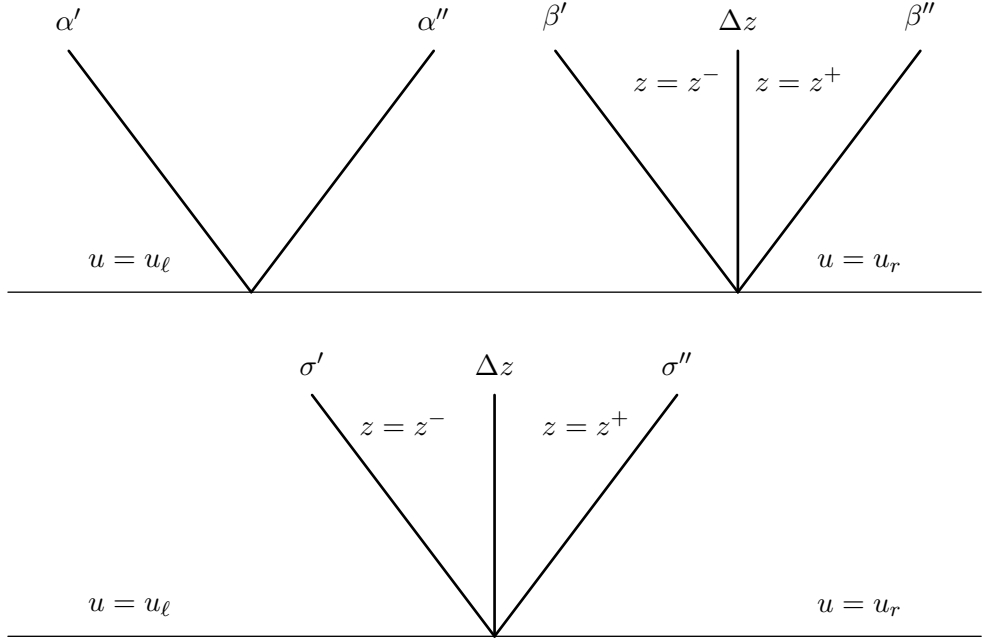
\begin{figure}[H]
  \centering
  \begin{tikzpicture}[line cap=round,line join=round,x=0.8cm,y=0.8cm]
    \draw [line width=0.5pt] (-8,0) -- (8,0); \draw [line width=0.5pt,
    line width=1pt] (-4,0) -- (-7,4); \draw [line width=0.5pt, line
    width=1pt] (-4,0) -- (-1,4); \draw [line width=0.5pt, line
    width=1pt] (4,0) -- (7,4); \draw [line width=0.5pt, line
    width=1pt] (4,0) -- (1,4); \draw [line width=0.5pt, line
    width=1pt] (4,0) -- (4,4); \draw (-1,4) node
    [label=above:{$\alpha''$}] {}; \draw (-7,4) node
    [label=above:{$\alpha'$}] {}; \draw (1,4) node
    [label=above:{$\beta'$}] {}; \draw (7,4) node
    [label=above:{$\beta''$}] {}; \draw (4,4) node
    [label=above:{$\Delta z$}] {}; \draw (-6,0) node
    [label=above:{$u=u_{\ell}$}] {}; \draw (6,0) node
    [label=above:{$u=u_{r}$}] {}; \draw (3,3) node
    [label=above:{$z=z^-$}] {};\draw (5,3) node
    [label=above:{$z=z^+$}] {};
  \end{tikzpicture}
  \begin{tikzpicture}[line cap=round,line join=round,x=0.8cm,y=0.8cm]
    \draw [line width=0.5pt] (-8,0) -- (8,0); \draw [line width=0.5pt,
    line width=1pt] (0,0) -- (3,4); \draw [line width=0.5pt, line
    width=1pt] (0,0) -- (-3,4); \draw [line width=0.5pt, line
    width=1pt] (0,0) -- (0,4); \draw (-3,4) node
    [label=above:{$\sigma'$}] {}; \draw (3,4) node
    [label=above:{$\sigma''$}] {}; \draw (0,4) node
    [label=above:{$\Delta z$}] {}; \draw (7,5) node
    [label=above:{$ $}] {}; \draw (-6,0) node
    [label=above:{$u=u_{\ell}$}] {}; \draw (6,0) node
    [label=above:{$u=u_{r}$}] {}; \draw (-1.5,3) node
    [label=above:{$z=z^-$}] {};\draw (1.5,3) node
    [label=above:{$z=z^+$}] {};
  \end{tikzpicture}
  \caption{\small Notation used in Lemma~\ref{lem:superRP}}
  \label{fig:bo}
\end{figure}

\begin{proof}
  Let $u_*$ be defined as in~\eqref{eq:11} and use the
  notation~\eqref{eq:39} to obtain:
  \begin{displaymath}
    \norma{\pmb\sigma - (\pmb\alpha+\pmb\beta)}
    \leq
    \norma{E (z^+,z^-,u_r,u_\ell) - E (z^+,z^-,u_*,u_\ell)}
    \leq
    \O \, \norma{u_r-u_*} \,.
  \end{displaymath}
  An application of Lemma~\ref{lem:inter} completes the proof.
\end{proof}

Lemma~\ref{lem:inter} and Lemma~\ref{lem:superRP} suggest that the
quantity $\norma{z^+-z^-}$ is a convenient way to measure the strength
of the zero--waves associated to the coupling condition. More
precisely, we define the strength of the zero--wave at a junction with
parameters $z^+, z^- \in \mathcal{Z}$ as $\sigma = \norma{z^+-z^-}$.

\begin{lemma}
  \label{lem:nonfis}
  Let $f$ satisfy~\ref{it:f1}--\ref{it:f4} and $\Xi$
  satisfy~\ref{eq:Xi1}, \ref{eq:Xi3}. Then, there exists a positive
  $\delta$ such that if $u_\ell,u_r \in \Omega$;
  $z^+, z^-\in \mathcal Z$ and $\pmb\alpha, \pmb\beta \in \reali^n$
  are such that
  \begin{displaymath}
    \norma{u_\ell-u_r} \leq \delta \,,\qquad
    \norma{z^+ - z^-} \leq \delta \,,\qquad
    \norma{\pmb\alpha} + \norma{\pmb\beta} \leq \delta
  \end{displaymath}
  with reference to Figure~\ref{figL:1}, the following general
  interaction estimate holds:
  \begin{displaymath}
    \modulo{\norma{u_r - u_*} - \norma{\hat u - u_\ell}}
    \leq
    \O \left(
      \sum_{j=1}^n  \modulo{\beta_j}
      +
      \norma{z^+ - z^-}
    \right)
    \norma{\hat u - u_\ell}
  \end{displaymath}
\end{lemma}

  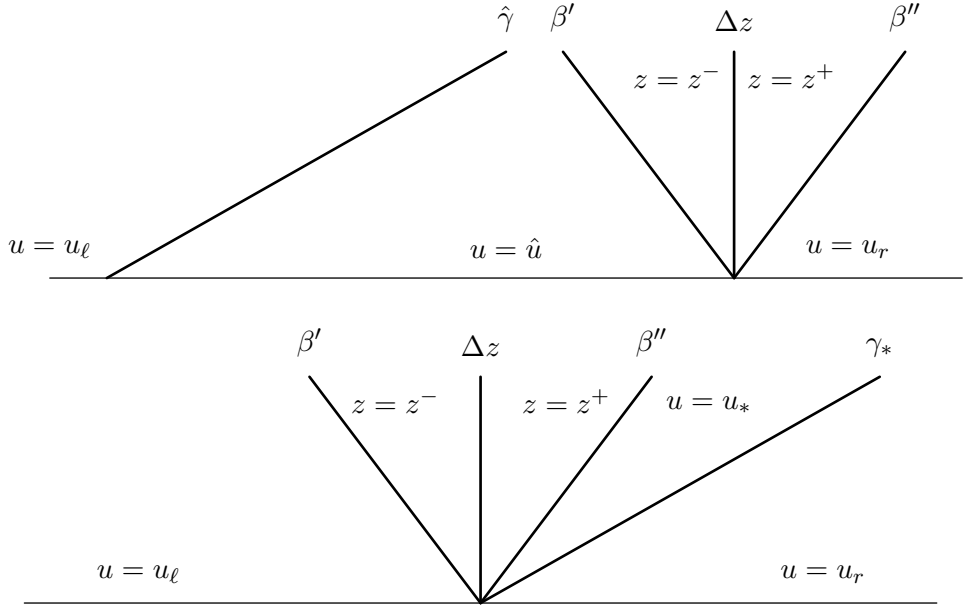
\begin{figure}[H]
    \centering
    \begin{tikzpicture}[line cap=round,line join=round,x=0.75cm,y=0.75cm]
      \draw [line width=0.5pt] (-8,0) -- (8,0); \draw [line
      width=0.5pt, line width=1pt] (-7,0) -- (0,4); \draw [line
      width=0.5pt, line width=1pt] (4,0) -- (7,4); \draw [line
      width=0.5pt, line width=1pt] (4,0) -- (1,4); \draw [line
      width=0.5pt, line width=1pt] (4,0) -- (4,4); \draw (0,4) node
      [label=above:{$\hat\gamma$}] {}; \draw (1,4) node
      [label=above:{$\beta'$}] {}; \draw (7,4) node
      [label=above:{$\beta''$}] {}; \draw (4,4) node
      [label=above:{$\Delta z$}] {}; \draw (-8,0) node
      [label=above:{$u=u_{\ell}$}] {}; \draw (0,0) node
      [label=above:{$u=\hat u$}] {}; \draw (6,0) node
      [label=above:{$u=u_{r}$}] {}; \draw (3,3) node
      [label=above:{$z=z^-$}] {};\draw (5,3) node
      [label=above:{$z=z^+$}] {};
    \end{tikzpicture}
    \begin{tikzpicture}[line cap=round,line join=round,x=0.75cm,y=0.75cm]
      \draw [line width=0.5pt] (-8,0) -- (8,0); \draw [line
      width=0.5pt, line width=1pt] (0,0) -- (3,4); \draw [line
      width=0.5pt, line width=1pt] (0,0) -- (-3,4); \draw [line
      width=0.5pt, line width=1pt] (0,0) -- (0,4); \draw [line
      width=0.5pt, line width=1pt] (0,0) -- (7,4); \draw (-3,4) node
      [label=above:{$\beta'$}] {}; \draw (3,4) node
      [label=above:{$\beta''$}] {}; \draw (0,4) node
      [label=above:{$\Delta z$}] {}; \draw (7,4) node
      [label=above:{$\gamma_{*}$}] {}; \draw (7,5) node
      [label=above:{$ $}] {}; \draw (-6,0) node
      [label=above:{$u=u_{\ell}$}] {}; \draw (6,0) node
      [label=above:{$u=u_{r}$}] {}; \draw (4,3) node
      [label=above:{$u=u_{*}$}] {}; \draw (-1.5,3) node
      [label=above:{$z=z^-$}] {};\draw (1.5,3) node
      [label=above:{$z=z^+$}] {};
    \end{tikzpicture}
    \caption{\small Notation used in Lemma~\ref{lem:nonfis}.}
    \label{figL:1}
  \end{figure}

\begin{proof}
  Referring to Figure~\ref{figL:1}, straightforward computations lead
  to:
  \begin{eqnarray}
    \nonumber
    &
    & \modulo{\norma{u_r - u_*} - \norma{u_\ell - \hat u}}
    \\
    \nonumber
    & \leq
    & \norma{ (u_r - u_*) - (u_\ell - \hat u) }
    \\
    \nonumber
    & =
    & \left\|
      \left(H (z^+,\beta'') \circ T (z^+,z^-)\circ H (z^-,\beta') (u_\ell) - u_\ell\right)
      \right.
    \\
    \nonumber
    &
    & -
      \left.
      \left(H (z^+,\beta'') \circ T (z^+,z^-)\circ H (z^-,\beta') (\hat u) - \hat u\right)
      \right\|
    \\
    \label{eq:13}
    & \leq
    & \left\|
      \left(H (z^+,\beta'') \circ T (z^+,z^-)\circ H (z^-,\beta') (u_\ell)
      -
      T (z^+,z^-) (u_\ell)\right)
      \right.
    \\
    \label{eq:14}
    &
    &
      \left.
      -
      \left(H (z^+,\beta'') \circ T (z^+,z^-)\circ H (z^-,\beta') (\hat u)
      -
      T (z^+,z^-) (\hat u)\right)
      \right\|
    \\
    \label{eq:15}
    &
    & +
      \norma{
      \left(T (z^+,z^-) (u_\ell) - u_\ell\right)
      -
      \left(T (z^+,z^-) (\hat u) - \hat u\right)}
    \\
    \nonumber
    & \leq
    & \O \left(
      \norma{u_\ell-\hat u} \sum_{j=1}^n \modulo{\beta_j}
      +
      \norma{u_\ell-\hat u} \, \norma{z^+-z^-}
      \right)
  \end{eqnarray}
  completing the proof. Above we used the fact that the term in the
  norm~\eqref{eq:13}--\eqref{eq:14} is a smooth function that vanishes
  for $u_\ell = \hat u$ as well as for $\beta=0$,
  see~\cite[\S~2.9]{Bressan2000}.  Moreover,
  Lemma~\ref{lem:EstimateXi2} can be applied to the
  term~\eqref{eq:15}, with $\phi (z^+,z^-,u) = T (z^+,z^-) (u) - u$.
\end{proof}

\begin{lemma}
  \label{lem:zetaPbR}
  Let $f$ satisfy~\ref{it:f1}--\ref{it:f4}, $\Xi$
  satisfy~\ref{eq:Xi1}--\ref{eq:Xi3}. Then, there exists a $\delta>0$
  such that if $\hat u^l, \hat u^r,\check u^l,\check u^r \in \Omega$,
  $\hat z^-, \hat z^-,\check z^+,\check z^+ \in \mathcal{Z}$ and
  \begin{displaymath}
    \norma{\hat z^+ - \hat z^-}
    + \norma{\hat u^l- \hat u^r}
    <
    \delta
    \,,\qquad  \qquad
    \norma{\check z^+ - \check z^-}
    + \norma{\check u^l- \check u^r}
    <
    \delta
  \end{displaymath}
  the solutions $\hat u$ and $\check u$ to the corresponding
  Generalized Riemann Problems~\eqref{eq:21} with data
  \begin{equation}
    \label{eq:71}
    \begin{array}{r@{\,}c@{\,}l@{\qquad\qquad}r@{\,}c@{\,}l}
      \hat u_o (x)
      & =
      & \left\{
        \begin{array}{cr@{\,}c@{\,}l}
          \hat u^l
          & x
          & <
          & \xi
          \\
          \hat u^r
          & x
          & \geq
          & \xi
        \end{array}
            \right.
          &  \hat \zeta (x)
          & =
          & \left\{
            \begin{array}{cr@{\,}c@{\,}l}
              \hat z^-
              & x
              & <
              & \xi
              \\
              \hat z^+
              & x
              & \geq
              & \xi
            \end{array}
                \right.
      \\
      \check u_o (x)
      & =
      & \left\{
        \begin{array}{cr@{\,}c@{\,}l}
          \check u^l
          & x
          & <
          & \xi
          \\
          \check u^r
          & x
          & \geq
          & \xi
        \end{array}
            \right.
          & \check \zeta (x)
          & =
          & \left\{
            \begin{array}{cr@{\,}c@{\,}l}
              \check z^-
              & x
              & <
              & \xi
              \\
              \check z^+
              & x
              & \geq
              & \xi
            \end{array}
                \right.
    \end{array}
  \end{equation}
  satisfy the estimate
  \begin{eqnarray*}
    &
    & \dfrac{1}{h}
      \int_{\xi-\hat \lambda \, h}^{\xi+\hat \lambda \, h}
      \norma{\hat u (h,x) - \check u (h,x)} \d{x}
    \\
    & \leq
    & \O
      \left(
      \norma{\hat u^l- \check u^l}
      + \norma{\hat u^r- \check u^r}
      + \norma{\hat u^r- \hat u^l}
      \left(
      \norma{\hat z^- - \check z^-}
      +
      \norma{\hat z^+ - \check z^+}
      \right)
      \right.
    \\
    &
    & \qquad\qquad
      \left.
      +
      \min\left\{
      \norma{\hat z^+ - \hat z^+}
      +
      \norma{\check z^+ - \check z^-}
      \,,\;
      \norma{\hat z^- - \check z^-}
      +
      \norma{\hat z^+ - \check z^+}
      \right\}
      \right) .
  \end{eqnarray*}
\end{lemma}

\begin{proof}
  The self similarity of the solutions to Riemann Problems ensures
  that
  \begin{displaymath}
    \dfrac{1}{h}
    \int_{\xi-\hat \lambda \, h}^{\xi+\hat \lambda \, h}
    \norma{\hat u (h,x) - \check u (h,x)} \d{x}
    =
    \int_{\xi-\hat \lambda}^{\xi+\hat \lambda}
    \norma{\hat u (1,\xi+\lambda) - \check u (1,\xi+\lambda)} \d{\lambda} \,.
  \end{displaymath}
  Recall that both $\lambda \mapsto \hat u (1,\xi+\lambda)$ and
  $\lambda \mapsto \check u (1,\xi+\lambda)$ consist of a sequence of
  constant states, jump discontinuities and Lipschitz continuous
  rarefaction profiles. Call
  $\hat p_1, \hat p_2, \ldots \hat p_{2n+2}$ the positions of waves in
  $\hat u$, in the sense that $\hat p_{2l-1} = \hat p_{2l}$ when a
  shock, a contact discontinuity or a zero wave in $\hat u$ is
  supported there; while $\hat p_{2l-1} < \hat p_l$ whenever a (non
  trivial) rarefaction in $\hat u$ is supported on
  $[\hat p_{2l-1}, \hat p_l]$.  Define
  $\check p_1, \check p_2, \ldots \check p_{2n+2}$ similarly, with
  reference to $\check u$.  The map $(z^-, z^+, u^l, u^r) \mapsto p$
  is Lipschitz in the $\hat z$ variables and smooth in the $u$
  variables.

  Set $\hat p_0 = \check p_0 = \xi-\hat \lambda$ and
  $\hat p_{2n+3} = \check p_{2n+3} = \xi+\hat \lambda$.
  Then,
  \begin{eqnarray*}
    &
    & \int_{\xi-\hat \lambda}^{\xi+\hat \lambda}
      \norma{\hat u (1,\xi+\lambda) - \check u (1,\xi+\lambda)} \d{\lambda}
    \\
    & \leq
    & \O \left(
      \sum_{i=1}^{2n+2}
      \modulo{\hat p_i - \check p_i}
      +
      \sum_{j=0}^{n+1}
      \norma{
      \hat u \left(1, \frac{\hat p_{2j}+\hat p_{2j+1}}{2}\right)
      -
      \check u \left(1, \frac{\check p_{2j}+\check p_{2j+1}}{2}\right)}
      \right.
    \\
    &
    & \qquad\qquad +
      \left.
      \sum_{j=1}^{n+1}
      \int_{[\hat p_{2j-1}, \hat p_{2j}] \cap [\check p_{2j-1}, \check p_{2j}]}
      \norma{\hat u (1,x) - \check u (1,x)} \d{x}
      \right) \,.
  \end{eqnarray*}
  Above, each of the quantities $\hat p_i - \check p_i$,
  $\hat u \left(1, \frac{\hat p_{2j}+\hat p_{2j+1}}{2}\right) - \check
  u \left(1, \frac{\check p_{2j}+\check p_{2j+1}}{2}\right)$ and
  $\hat u (1,x) - \check u (1,x)$ can be written as a difference
  $G (\hat z^-, \hat z^+, \hat u^l, \hat u^r) - G (\check z^-, \check
  z^+, \check u^l, \check u^r)$, the function $G$ being Lipschitz
  continuous in $z$ and smooth in $u$. Hence,
  \begin{eqnarray*}
    &
    & \norma{
      G (\hat z^-, \hat z^+, \hat u^l, \hat u^r) -
      G (\check z^-, \check z^+, \check u^l, \check u^r)}
    \\
    & \leq
    & \norma{
      G (\hat z^-, \hat z^+, \hat u^l, \hat u^r) -
      G (\check z^-, \check z^+, \hat u^l, \hat u^r)}
      +
      \norma{
      G (\check z^-, \check z^+, \hat u^l, \hat u^r) -
      G (\check z^-, \check z^+, \check u^l, \check u^r)}
    \\
    & \leq
    & \norma{
      G (\hat z^-, \hat z^+, \hat u^l, \hat u^r) -
      G (\check z^-, \check z^+, \hat u^l, \hat u^r)} +
      \O\left(
      \norma{\check u^l - \hat u^l} +
      \norma{\check u^r - \hat u^r}
      \right) \,.
  \end{eqnarray*}
  Moreover,
  \begin{eqnarray*}
    &
    & \norma{
      G (\hat z^-, \hat z^+, \hat u^l, \hat u^r) -
      G (\check z^-, \check z^+, \hat u^l, \hat u^r)}
    \\
    & \leq
    & \norma{
      \left(
      G (\hat z^-, \hat z^+, \hat u^l, \hat u^r) -
      G (\check z^-, \check z^+, \hat u^l, \hat u^r)
      \right)
      -
      \left(
      G (\hat z^-, \hat z^+, \hat u^l, \hat u^l) -
      G (\check z^-, \check z^+, \hat u^l, \hat u^l)
      \right)
      }
    \\
    &
    & +
      \norma{
      \left(
      G (\hat z^-, \hat z^+, \hat u^l, \hat u^l) -
      G (\check z^-, \check z^+, \hat u^l, \hat u^l)
      \right)
      -
      \underbrace{
      \left(
      G (\hat z^-, \hat z^-, \hat u^l, \hat u^l) -
      G (\check z^-, \check z^-, \hat u^l, \hat u^l)
      \right)}_{=0}
      }
    \\
    & \leq
    & \modulo{
      \int_{\hat u^l}^{\hat u^r}
      \norma{
      D_4 G (\hat z^-, \hat z^+, \hat u^l, w) -
      D_4 G (\check z^-, \check z^+, \hat u^l, w)
      }\d{w}
      }
    \\
    &
    & + \O \min\left\{
      \norma{\hat z^+ - \hat z^+}
      +
      \norma{\check z^+ - \check z^-}
      \,,\;
      \norma{\hat z^- - \check z^-}
      +
      \norma{\hat z^+ - \check z^+}
      \right\}
    \\
    & \leq
    & \O \norma{\hat u^r - \hat u^l}
      \left(
      \norma{\hat z^- - \check z^-}
      +
      \norma{\hat z^+ - \check z^+}
      \right)
    \\
    &
    & + \O \min\left\{
      \norma{\hat z^+ - \hat z^+}
      +
      \norma{\check z^+ - \check z^-}
      \,,\;
      \norma{\hat z^- - \check z^-}
      +
      \norma{\hat z^+ - \check z^+}
      \right\} \,,
  \end{eqnarray*}
  completing the proof.
\end{proof}

\subsection{The Case $\zeta \in \PC (\reali; \reali^p)$}

\subsubsection{Wave Front Tracking}
Fix a $\zeta \in (\PC \cap \BV) (\reali; \reali^p)$, $\mathcal{I} (z)$
being the set of points of jump in $z$. Let
$u \in \PC (\reali; \Omega)$ and call $\mathcal{I} (u)$ the set of
points of jump in $u$.  Let $\sigma_{x,i}$ be the (signed) strength of
the $i$--th wave in the solution to the Riemann problem
for~\eqref{eq:21} with data $u(x-)$ and $u(x+)$,
i.e.~$(\sigma_{ x,1} , \ldots, \sigma_{ x,n} ) = E \left(\zeta (x+),
  \zeta (x-), u(x+), u(x-)\right)$ as in~\eqref{eq:39}.  Define
\begin{equation}
  \label{eq:5}
  \mathcal{A}^\zeta (u)
   =
  \left\{
    \begin{array}{l}
      \left((x,i),(y,j)\right) \in
      \left(\left(\mathcal{I} (u) \cup \mathcal{I}(\zeta)\right)
      \times \{1, \ldots, n\}\right)^2
      \colon
      \\
      x<y \mbox{ and either } i>j
      \mbox{ or $i=j$, the $i$--th field is}
      \\
      \mbox{genuinely non linear and } \min\{\sigma_{x,i}, \sigma_{y,i}\} < 0
    \end{array}
  \right\}
\end{equation}
Extending what introduced in~\cite{ColomboGuerra2008}, the linear and
the interaction potential are
\begin{eqnarray}
  \nonumber
  \mathbf{V}^\zeta(u)
  &  =
  & \sum_{x\in \mathcal{I}(u) \cup \mathcal{I} (\zeta)}
    \sum_{i=1}^n
    \modulo{\sigma_{x,i}}
    +
    \sum_{x\in \mathcal{I}(\zeta)}
    \norma{\Delta\zeta (x)}
  \\
  \nonumber
  \mathbf{Q}^\zeta(u)
  &  =
  & \sum_{\left((x,i),(y,j)\right) \in \mathcal{A}^\zeta(u)}
    \modulo{\sigma_{x,i}\sigma_{y,j}}
  \\
  \nonumber
  &
  & \quad
    +
    \sum_{x \in \mathcal{I} (\zeta)}
    \norma{\Delta\zeta (x)}
    \left(
    \sum_{y \in \mathcal{I} (u)\,,\, y<x} \; \sum_{j > i_0} \modulo{\sigma_{y,j}}
    +
    \sum_{y \in \mathcal{I} (u)\,,\, y>x} \; \sum_{j \leq i_0} \modulo{\sigma_{y,j}}
    \right)
  \\
  \label{eq:4}
  \mathbf{\Upsilon}^\zeta (u)
  &  =
  & \mathbf{V}^\zeta (u) + C_0 \cdot \mathbf{Q}^\zeta (u)
\end{eqnarray}
where $C_0$ is a suitable positive constant. For $\delta>0$
sufficiently small, we define
\begin{equation}
  \label{def:2.6}
  \mathcal{D}_\delta^\zeta
   =
  \mathrm{cl} \left\{
    u \in \bar{u} + \L1\left(\reali,\Omega \right) \colon
    u \hbox { is piecewise constant and } \mathbf{\Upsilon}^\zeta(u) < \delta
  \right\}
\end{equation}
where the closure is in the strong $\L1$--topology.

We adapt the wave-front tracking techniques
from~\cite{AmadoriGosseGuerra2002, Bressan2000, ColomboGuerraHolle,
  ColomboMarcellini2010, GuerraMarcelliniSchleper2009} to construct a
sequence of approximate solutions to the Cauchy problem~\eqref{eq:21}
and prove uniform $\BV$-estimates in space. The approximate solutions
converge towards a solution to the Cauchy problem with finitely many
junctions. First, we define the approximations.

\begin{definition}
  \label{def:WFTApproxSol}
  Let $\zeta \in \BV (\reali; \mathcal{Z})$ be piecewise constant.
  For $\epsilon>0$, a continuous map
  \begin{equation*}
    u^\epsilon \colon \left[0, +\infty\right[ \to \Lloc1(\reali; \reali^n)
  \end{equation*}
  is an $\epsilon$-approximate solution to~\eqref{eq:21} if the
  following conditions hold:
  \begin{itemize}
  \item $u^\epsilon$, as a function of $(t,x)$, is piecewise constant
    with discontinuities along finitely many straight lines in the
    $(t,x)$-plane. There are only finitely many wave-front
    interactions and at most two waves interact with each other. There
    are four types of discontinuities: shocks (or contact
    discontinuities), rarefaction waves, non--physical waves and
    zero--waves. We distinguish these waves' indexes in the sets
    $\mathcal J =\mathcal{S} \cup \mathcal{R} \cup
    \mathcal{N}\mathcal{P} \cup \mathcal{Z}\mathcal{W}$, the generic
    index in $\mathcal{J}$ being $\alpha$.
  \item At a shock (or contact discontinuity) $x_\alpha=x_\alpha(t)$,
    $\alpha\in\mathcal S$, the traces $u^+=u^\epsilon (t,x_\alpha+)$
    and $u^-=u^\epsilon (t,x_\alpha-)$ are related by
    $u^+=H_{i_\alpha}(\zeta(x_\alpha),\sigma_\alpha)(u^-)$ for
    $1\le i_\alpha\le n$, see~\eqref{eq:6}. If the $i_\alpha$--th
    family is genuinely nonlinear, the admissibility condition
    $\sigma_\alpha<0$ holds and
    \begin{equation}
      \label{eq:59}
      \modulo{\dot x_\alpha -\lambda_{i_\alpha}(\zeta(x_\alpha),u^+,u^-)}
      \le \epsilon,
    \end{equation}
    where $\lambda_{i_\alpha}(\zeta(x_\alpha),u^+,u^-)$ is the wave
    speed described by the Rankine-Hugoniot conditions w.r.t.\
    $u\mapsto f(\zeta(x_\alpha),u)$.
  \item On the sides of a rarefaction wave $x_\alpha = x_\alpha(t)$,
    $\alpha\in\mathcal R$ in a genuinely nonlinear family, the traces
    are related by
    $u^+ = H_{i_\alpha}(\zeta(x_\alpha),\sigma_\alpha)(u^-)$ where
    $1 \le i_\alpha \le n$ and $0 < \sigma_\alpha \le
    \epsilon$. Moreover,
    \begin{equation*}
      \modulo{\dot x_\alpha -\lambda_{i_\alpha}(\zeta(x_\alpha),u^+)} \le \epsilon.
    \end{equation*}
  \item All non--physical fronts $x = x_\alpha(t)$,
    $\alpha \in \mathcal{NP}$ travel at the same speed
    $\dot x_\alpha = \hat\lambda$ with
    $\hat\lambda > \sup_{z,u,i} \modulo{\lambda_i (z,u)}$. The total
    strength of all non--physical fronts is uniformly bounded by
    \begin{equation*}
      \sum_{\alpha\in\mathcal{NP}}
      \norma{u^\epsilon(t,x_\alpha +)-u^\epsilon(t,x_\alpha -)}
      \le
      \epsilon
      \quad \mbox{for all } t > 0 \,.
    \end{equation*}
  \item Zero--waves are located at the discontinuities
    $x_\alpha \in \mathcal{I} (\zeta)$. At a zero--wave $x_\alpha$,
    $\alpha \in \mathcal{ZW}$, the traces are related by the coupling
    condition
    $u^+=T\left(\zeta (x_\alpha+), \zeta (x_\alpha)\right)(u^-)$ for
    all $t>0$, see~\eqref{eq:1}, except at the interaction times.
  \item At time $t=0$, $u^\epsilon$ satisfies
    $\norma{u^\epsilon(0,\cdot)-u_o}_{\L1(\reali;\reali^n)} \le
    \epsilon$.
  \end{itemize}
\end{definition}

\begin{proposition}[{\cite[Theorem~4.11]{ColomboGuerraHolle}}]
  \label{thm:ExistenceEpsilonApprox}
  Let $f$ satisfy~\ref{it:f1}--\ref{it:f4} and $\Xi$
  satisfy~\ref{eq:Xi1}--\ref{eq:Xi3}. Fix $\bar z \in \mathcal{Z}$ and
  $\bar u \in \Omega$. Then, there exist $\delta, K > 0$ such that for
  all piecewise constant $\zeta \in \BV (\reali; \mathcal{Z})$ with
  \begin{equation}
    \label{eq:23}
    \zeta (\reali) \subseteq B (\bar z; \delta) \quad \mbox{ and } \quad
    \tv (\zeta) < \delta
  \end{equation}
  and for all initial data $u_o \in \mathcal{D}^\zeta_\delta$,
  for every $\epsilon$ sufficiently small there exists an
  $\epsilon$--approximate solution to~\eqref{eq:21} in the sense of
  Definition~\ref{def:WFTApproxSol}. Moreover, the total variation in
  space $\tv(u^\epsilon(t,\cdot))$ and the total variation in time
  $\tv(u^\epsilon(\cdot,x))$ are bounded uniformly for $\epsilon$
  sufficiently small, i.e., for all $t>0$ and for all $x \in \reali$
  \begin{displaymath}
    \mathbf \Upsilon^\zeta (u^\epsilon(t,\cdot)) \leq \delta + K \, \epsilon
    \quad \mbox{ and } \quad
    \tv(u^\epsilon(\cdot,x)) \leq K \,.
  \end{displaymath}
\end{proposition}

\begin{proofof}{Proposition~\ref{thm:ExistenceEpsilonApprox}}
  We use here the well known wave front tracking algorithm originally
  introduced in~\cite{DafermosWFT} and adapted to the present
  situation in~\cite{AmadoriGosseGuerra2002}, see
  also~\cite{BaitiJenssen1998, Bressan2000, BressanLiuYang1999,
    Dafermos2000, MR1912206}. Indeed, waves supported in the points of
  jump of $\zeta$, that is the \emph{zero} waves indexed in
  $\mathcal{Z}\mathcal{W}$, behave as linearly degenerate waves from
  the point of view of the wave front tracking algorithm developed
  in~\cite{BaitiJenssen1998}, to which we refer also for the
  terminology. Remark that Lemma~\ref{lem:inter},
  Lemma~\ref{lem:superRP} and Lemma~\ref{lem:nonfis} allow to extend
  to interactions involving zero waves estimates of the same form as
  those typically used in standard wave front tracking procedures.

  We refer to~\cite[Theorem~4.11]{ColomboGuerraHolle} for more
  details.
\end{proofof}

\subsubsection{An Extended \emph{Almost--Decreasing} Functional}

To prove the Lipschitz continuous dependence of solutions on the
initial datum, we introduce a functional uniformly equivalent to the
$\L 1(\reali,\reali^n)$--distance~\cite{BressanLiuYang1999}. We follow
the considerations in~\cite[Section~4.2]{AmadoriGosseGuerra2002}.

Let $u$, respectively $v$, be an $\epsilon$--approximate, respectively
$\epsilon'$--approximate, solutions as in
Proposition~\ref{thm:ExistenceEpsilonApprox} with the same piecewise
constant $\zeta \in \BV(\reali;\mathcal Z)$ as in~\eqref{eq:24}. The
functions $u (0, \cdot)$ and $v (0,\cdot)$ do not necessarily
coincide.  Introduce the concatenation of shock curves
\begin{equation}
  \label{eq:3}
  S(z,\pmb{q})(u)
   =
  S_{n}(z,q_{n})
  \circ \dots \circ
  S_1(z,q_1)(u)
\end{equation}
where $q_i \mapsto S_i(z,q_i)$ are the shock curves with respect to
the flux function $u \mapsto f(z,u)$ possibly violating the
admissibility condition.  We define
$\mathbf{q} (z,u,v) \equiv (\mathrm{q}_1, \ldots,
\mathrm{q}_n)(z,u,v)$ implicitly by
\begin{equation}
  \label{eq:69}
  v
  =
  S \left(z,\mathbf q (z,u,v) \right) (u)
\end{equation}
and the $i$--th shock speed, with the same notation as
in~\eqref{eq:59},
\begin{equation}
  \label{eq:70}
  \begin{array}{rcl}
    \Lambda_i (z,u,v)
    &  =
    & \lambda_i
      \left(
      z,
      S_{i}\left(z,\mathrm{q}_{i} (z,u,v)\right)
      \circ \dots \circ
      S_1\left(z,\mathrm{q}_1 (z,u,v)\right) (u)
      ,
      \right.
    \\
    &
    & \qquad\qquad
      \left.
      S_{i-1}\left(z,\mathrm{q}_{i-1} (z,u,v)\right)
      \circ \dots \circ
      S_1\left(z, \mathrm{q}_1 (z,u,v)\right) (u)
      \right) \,.
    \\
  \end{array}
\end{equation}
For sufficiently small $\mathbf{q} (z,u,v)$ and for $z$ in a small
neighborhood of $\bar z$, we have
\begin{displaymath}
  \frac{1}{C} \, |u-v|
  \le
  \sum_{i=1}^n \modulo{\mathrm{q}_i (z,u,v)}
  \le C \, |u-v|
\end{displaymath}
for a constant $C > 1$.  We define the following functional equivalent
to the $\L 1(\reali;\reali^n)$ distance:
\begin{eqnarray}
  \nonumber
  q_i (t,x)
  &  =
  & \mathrm{q}_i\left(\zeta (x), u (t,x), v (t,x)\right)
    \qquad i=1, \ldots, n \,,
  \\
  \label{eq:36}
  \Phi(u,v)(t)
  &  =
  & \sum_{i=1}^n\int_{\reali} \modulo{q_i(t,x)} \, W_i(t,x)\d x \,,
  \\
  \label{eq:72}
  W_i(t,x)
  &  =
  & 1
    +
    \kappa_1 \, B_i(t,x)
    +
    \kappa_2 \left( Q(u,t)+Q(v,t)\right),
  \\
  \label{eq:17}
  B_i(t,x)
  &  =
  & A_i(t,x)
    +
    \begin{cases}
      \sum_{x_\alpha < x,\alpha\in\mathcal {ZW}} |\sigma_\alpha|
      \qquad &\mbox{ if } i\le i_o \,,
      \\
      \sum_{x_\alpha > x,\alpha\in\mathcal {ZW}} |\sigma_\alpha|
      \qquad &\mbox{ if } i> i_o \,,
    \end{cases}
\end{eqnarray}
with positive $\kappa_1,\kappa_2$, chosen below and with $A_i$ defined
as in~\cite[Formula~(4.9)]{AmadoriGosseGuerra2002},
\cite[Formul\ae~(8.8)--(8.9)]{Bressan2000}
or~\cite[Formul\ae~(2.17)--(2.18)]{BressanLiuYang1999} and $Q$ is the
usual Glimm interaction potential for piecewise constant approximate
solutions, see~\cite[Formula~(7.54)]{Bressan2000}, also including all
zero waves. We follow~\cite{Bressan2000,BressanLiuYang1999} and ensure
that whatever the values of the constants $\kappa_1,\kappa_2$, the
parameter $\delta > 0$ in~\eqref{eq:23} can be reduced so that
$1 \leq W_i(t,x) \leq 2$.

We obtain the following result by the same procedure as in the proof
of~\cite[Lemma~9]{AmadoriGosseGuerra2002}. Observe that all arguments
hold for $f(\zeta(x),\cdot)$ instead of $f$ by~\ref{it:f1} and by the
smallness of $\tv(\zeta)$.

\begin{lemma}
  \label{lem:pocosu}
  Let $f$ satisfy~\ref{it:f1}--\ref{it:f4} and $\Xi$
  satisfy~\ref{eq:Xi1}--\ref{eq:Xi3}. Fix $\bar z \in \mathcal{Z}$ and
  $\bar u \in \Omega$.  There exist suitable positive
  $\kappa_1,\kappa_2,\delta$ such that if $\zeta$ is piecewise
  constant and satisfies~\eqref{eq:23}, $u$ is an
  $\epsilon$--approximate solution and $v$ is an
  $\epsilon'$--approximate solution as in
  Theorem~\ref{thm:ExistenceEpsilonApprox}, both corresponding to
  $\zeta$, with $u (0, \cdot)$ and $v (0,\cdot)$ in
  $\mathcal{D}_\delta^\zeta$, as defined in~\eqref{def:2.6}, then the
  functional $\Phi$ satisfies for all $0\le s\le t$
  \begin{equation*}
    \Phi(u,v) (t)
    -
    \Phi(u,v) (s)
    \leq
    \O \; \max\{\epsilon,\epsilon'\} \; (t-s) \,.
  \end{equation*}
\end{lemma}

\begin{proof}
  At any interaction time $t$, the same computations as
  in~\cite{AmadoriGosseGuerra2002, Bressan2000, BressanLiuYang1999}
  ensure that $\Phi$ strictly decreases, thanks to the term
  $\kappa_2 \left( Q(u,t)+Q(v,t)\right)$ in~\eqref{eq:72}.

  At a time $t$ between any two interaction times, use the set
  $\mathcal{J}$ to index the discontinuities in $u (t,\cdot)$,
  $v (t, \cdot)$ and in $\zeta$ at time $t$. The same procedure used
  in~\cite{AmadoriGosseGuerra2002, Bressan2000, BressanLiuYang1999},
  to which we refer also for the standard notation employed below,
  allows to compute the derivative of $\Phi$ with respect to time as
  \begin{displaymath}
    \dfrac{\d{~}}{\d{t}} \Phi (u,v) (t)
    =
    \sum_{\nu \in \mathcal{J}}
    \sum_{i=1}^n
    \left(
      \modulo{q_i^{\nu+}} W_i^{\nu+} (\lambda_i^{\nu+} - \dot x_\nu)
      -
      \modulo{q_i^{\nu-}} W_i^{\nu-} (\lambda_i^{\nu-} - \dot x_\nu)
    \right) \,.
  \end{displaymath}
  The standard procedure in~\cite{AmadoriGosseGuerra2002, Bressan2000,
    BressanLiuYang1999} ensures that the above sum restricted to
  physical or non--physical waves is bounded as follows:
  \begin{displaymath}
    \sum_{\nu \in \mathcal{J} \setminus \mathcal{Z}\mathcal{W}} \;
    \sum_{i=1}^n
    \left(
      \modulo{q_i^{\nu+}} W_i^{\nu+} (\lambda_i^{\nu+} - \dot x_\nu)
      -
      \modulo{q_i^{\nu-}} W_i^{\nu-} (\lambda_i^{\nu-} - \dot x_\nu)
    \right)
    \leq
    \O \, \epsilon \,,
  \end{displaymath}
  where, as in Definition~\ref{def:WFTApproxSol},
  $\mathcal{Z}\mathcal{W}$ groups the indexes referring to zero waves.

  Now consider zero waves: $\nu \in \mathcal{Z}\mathcal{W}$. Then,
  $\dot x_\nu = 0$ and
  \begin{eqnarray}
    \nonumber
    &
    & \modulo{q_i^{\nu+}} W_i^{\nu+} (\lambda_i^{\nu+} - \dot x_\nu)
      -
      \modulo{q_i^{\nu-}} W_i^{\nu-} (\lambda_i^{\nu-} - \dot x_\nu)
    \\
    \label{eq:18}
    & =
    & \!\!
      W_i^{\nu+} \! \left(
      \modulo{q_i^{\nu+}}
      {-}
      \modulo{q_i^{\nu-}} \!
      \right)
      \lambda_i^{\nu+}
      {+}
      W_i^{\nu+}
      \modulo{q_i^{\nu-}}
      \left(
      \lambda_i^{\nu+} {-} \lambda_i^{\nu-}
      \right)
      {+}
      \left(W_i^{\nu+} {-} W_i^{\nu-}\right) \!
      \modulo{q_i^{\nu-}} \lambda_i^{\nu-} .
  \end{eqnarray}
  Now we bound the latter three summands separately. First, we use
  Lemma~\ref{lem:EstimateXi2} with
  $w \equiv (q_1^{\nu-}, \ldots, q_n^{\nu-})$ and
  $\phi (z^+,z^-,w) = q_i^{\nu+} - q_i^{\nu-}$, obtaining
  \begin{displaymath}
    \modulo{
      \modulo{q_i^{\nu+}}
      -
      \modulo{q_i^{\nu-}}
    }
    \leq
    \modulo{q_i^{\nu+} - q_i^{\nu-}}
    =
    \modulo{\phi (z^+, z^-, w) - \phi (z^+,z^-,0)}
    \leq
    \O \, \norma{\Delta z} \, \sum_{i=1}^n \modulo{q_i^{\nu-}} \,.
  \end{displaymath}
  Second, by the Lipschitz continuity of $\lambda_i$,
  \begin{displaymath}
    \modulo{q_i^{\nu-}} \, \modulo{\lambda_i^{\nu+} - \lambda_i^{\nu-}}
    \leq
    \O \norma{\Delta z} \, \sum_{i=1}^n \modulo{q_i^{\nu-}} \,.
  \end{displaymath}
  To bound the third term, introduce the sets
  $\hat I = \left\{i \in \{1, \ldots,n\} \colon q_i^{\nu+} \,
    q_i^{\nu-} > 0\right\}$ and
  $\check I = \left\{i \in \{1, \ldots,n\} \colon q_i^{\nu+} \,
    q_i^{\nu-} \leq 0\right\}$. For $i \in \hat I$ we have
  $A_i^+ = A_i^-$. If $i \leq i_o$ then $\lambda_i < 0$ and
  by~\eqref{eq:17} $\Delta B_i = \norma{\Delta z}$ while if
  $i \geq i_o+1$ then $\lambda_i>0$ and
  $\Delta B_i = - \norma{\Delta z}$. In both cases, the third summand
  in~\eqref{eq:18} satisfies
  \begin{eqnarray*}
    (W_i^{\nu+} - W_i^{\nu-}) \modulo{q_i^{\nu-}} \lambda_i^{\nu-}
    <
    - c \, \kappa_1 \, \norma{\Delta z} \, \modulo{q_i^{\nu-}}\,.
  \end{eqnarray*}
  On the other hand, if $i \in \check I$, $A_i^+$ and $A_i^-$ are not
  directly related, but
  \begin{equation}
    \label{eq:73}
    \modulo{q_i^{\nu+}} + \modulo{q_i^{\nu-}}
    =
    \modulo{q_i^{\nu+} - q_i^{\nu-}}
    \leq
    \O \, \norma{\Delta z} \, \sum_{i=1}^n \modulo{q_i^{\nu-}}
  \end{equation}
  so that
  \begin{equation}
    \label{eq:42}
    (W_i^{\nu+} - W_i^{\nu-}) \modulo{q_i^{\nu-}} \lambda_i^{\nu-}
    \leq
    \O \, \norma{\Delta z} \, \sum_{i=1}^n \modulo{q_i^{\nu-}} \,.
  \end{equation}
  Concluding the estimates on the three terms in~\eqref{eq:18}.
  Moreover, by~\eqref{eq:73},
  \begin{displaymath}
    \sum_{i=1}^n \modulo{q_i^{\nu-}}
    =
    \sum_{i \in \hat I} \modulo{q_i^{\nu-}}
    +
    \sum_{i \in \check I} \modulo{q_i^{\nu-}}
    \leq
    \sum_{i \in \hat I} \modulo{q_i^{\nu-}}
    +
    \O \, \norma{\Delta z} \, \sum_{i=1}^n \modulo{q_i^{\nu-}} \,.
  \end{displaymath}
  Hence, for $\Delta z$ sufficiently small,
  \begin{displaymath}
    \sum_{i=1}^n \modulo{q_i^{\nu-}}
    \leq
    2 \sum_{i \in \hat I} \modulo{q_i^{\nu-}} \,.
  \end{displaymath}
  Adding the different estimates obtained, we bound the term
  in~\eqref{eq:18} by
  \begin{eqnarray}
    \label{eq:45}
    &
    & \sum_{i=1}^n
      \left(
      \modulo{q_i^{\nu+}} W_i^{\nu+} (\lambda_i^{\nu+} - \dot x_\nu)
      -
      \modulo{q_i^{\nu-}} W_i^{\nu-} (\lambda_i^{\nu-} - \dot x_\nu)
      \right)
    \\
    \nonumber
    & \leq
    & - c \, \kappa_1 \, \norma{\Delta z}
      \sum_{i \in \hat I} \modulo{q_i^{\nu-}}
      +
      \O \, \norma{\Delta z} \, \sum_{i=1}^n \modulo{q_i^{\nu-}}
    \\
    \nonumber
    & \leq
    & (\O - c \, \kappa_1) \norma{\Delta z}
      \sum_{i \in \hat I} \modulo{q_i^{\nu-}}
    \\
    \nonumber
    & <
    & 0 \,,
  \end{eqnarray}
  assuming $\kappa_1$ sufficiently large.

  The proof is now completed by means of standard arguments, refer
  to~\cite{AmadoriGosseGuerra2002, Bressan2000, BressanLiuYang1999}.
\end{proof}

\subsubsection{Proof of Theorem~\ref{thm:sgrp} in the Case
  $\zeta \in \PC (\reali; \reali^p)$}
\label{subs:proof--theor-refthm:s}

Let $f$ satisfy~\ref{it:f1}--\ref{it:f4}, $\Xi$
satisfy~\ref{eq:Xi1}--\ref{eq:Xi4}. Fix $\bar z \in \mathcal{Z}$,
$\bar u \in \Omega$ and $\delta$ as defined in
Lemma~\ref{lem:pocosu}. Choose $\zeta \in \PC (\reali; \mathcal{Z})$
with $\tv (\zeta) < \delta$,
$\norma{\zeta - \bar z}_{\L\infty (\reali; \reali^p)} < \delta$ and
let $\mathcal{D}^\zeta  =  \mathcal{D}^\zeta_\delta$ be as
in~\eqref{def:2.6}. Note that
$\mathcal{D}^\zeta \subseteq \bar u + \L1 (\reali; \Omega)$ contains
all functions $u$ in $\bar u + \L1 (\reali; \Omega)$ with
$\tv (u) < \delta$.

\begin{lemma}
  \label{lem:123}
  There exist a positive $L$ and a unique semigroup
  $S^\zeta \colon \reali_+ \times \mathcal{D}^\zeta \to
  \mathcal{D}^\zeta$, obtained as limit of wave front tracking
  $\epsilon$--approximate solutions, such that points~1.,~2.~and~3.~in
  Theorem~\ref{thm:sgrp} hold.
\end{lemma}

\begin{proof}
  Since $\mathcal{D}^\zeta_\delta$ is separable, the existence of a
  Lipschitz continuous semigroup $S^\zeta$ enjoying properties~1.,
  2.~and~3.~can be obtained through the limit of (subsequences of)
  wave front tracking $\epsilon$--approximations in
  Definition~\ref{def:WFTApproxSol} following standard arguments, see
  for instance, \cite{AmadoriGosseGuerra2002} or~\cite{Bressan2000,
    BressanLiuYang1999, ColomboGuerra2008, ColomboGuerraHolle}.

  To prove uniqueness, let $\Sigma^\zeta$ be any Lipschitz continuous
  semigroups satisfying~1., 2.~and~3.. Fix an initial datum
  $u_o \in \mathcal{D}^\zeta_\delta$.  Call $u^\epsilon$ a wave front
  tracking $\epsilon$--approximate solution approaching the orbit
  $t \mapsto S^\zeta u_o$ as $\epsilon \to 0$. Then, by the Lipschitz
  continuity of $\Sigma^\zeta$ and Lemma~\ref{lem:pocosu},
  \begin{equation}
    \label{eq:65}
    \norma{\Sigma^\zeta_t u_o - S^\zeta_t u_o}_{\L1 (\reali; \reali^n)}
    \leq
    \norma{\Sigma^\zeta_t u^\epsilon (0) - u^\epsilon (t)}_{\L1 (\reali; \reali^n)}
    +
    \O \, \epsilon \, (1+t) \,.
  \end{equation}
  We now use~\cite[Theorem~2.9]{Bressan2000} to bound the first term
  in the right hand side above:
  \begin{equation}
    \label{eq:66}
    \norma{\Sigma^\zeta_t u^\epsilon (0) - u^\epsilon (t)}_{\L1 (\reali; \reali^n)}
    \leq
    \O
    \int_0^t \liminf_{h\to 0}
    \dfrac{\norma{\Sigma^\zeta_h u^\epsilon (t) - u^\epsilon (t+h)}_{\L1 (\reali; \reali^n)}}{h} \d\tau \,.
  \end{equation}
  Using the notation as in Definition~\ref{def:WFTApproxSol} and the
  classical estimates on physical and non--physical waves
  in~\cite[Lemma~9.1]{Bressan2000}, for $h$ so small that solutions to
  adjacent Riemann Problems do not overlap, we have:
  \begin{eqnarray*}
    &
    & \norma{\Sigma^\zeta_h u^\epsilon (t) - u^\epsilon (t+h)}_{\L1 (\reali; \reali^n)}
    \\
    & \leq
    & \sum_{\alpha \in \mathcal{J}}
      \int_{x_\alpha - \hat\lambda h}^{x_\alpha + \hat\lambda h}
      \norma{\left(\Sigma^\zeta_h u^\epsilon (t)\right) (x) - u^\epsilon (t+h,x)}
      \d{x}
    \\
    & \leq
    & \O \, \epsilon \, h
      +
      \sum_{\alpha \in \mathcal{ZW}}
      \int_{x_\alpha - \hat\lambda h}^{x_\alpha + \hat\lambda h}
      \norma{\left(\Sigma^\zeta_h u^\epsilon (t)\right) (x) - u^\epsilon (t+h,x)}
      \d{x}
    \\
    & =
    & \O \, \epsilon \, h
  \end{eqnarray*}
  since for all zero waves
  $\left(\Sigma^\zeta_h u^\epsilon (t)\right) (x) = u^\epsilon
  (t+h,x)$ for a.e.~$t$ and wave front tracking
  $\epsilon$--approximation solves Riemann Problems at zero waves
  exactly.

  Insert the latter bound in~\eqref{eq:66}, so that in the limit
  $\epsilon\to 0$, \eqref{eq:65} and the arbitrariness of $u_o$ yield
  the equality of $S^\zeta$ and $\Sigma^\zeta$.
\end{proof}

\begin{lemma}
  \label{lem:1}
  Fix $\xi \in \reali$ and define
  \begin{equation}
    \label{eq:47}
    \tilde\zeta (x) = \left\{
      \begin{array}{ll}
        \zeta (\xi)
        & x\leq \xi,
        \\
        \zeta (\xi+)
        & x> \xi .
      \end{array}
    \right.
  \end{equation}
  Choose $u \in \mathcal{D}^\zeta \cap \mathcal{D}^{\tilde\zeta}$.
  Then, for all $\theta > 0$,
  \begin{displaymath}
    \dfrac1\theta
    \int_{\xi - \hat\lambda \theta}^{\xi + \hat\lambda \theta}
    \norma{S^\zeta_\theta u (x)   - S^{\tilde\zeta}_\theta u (x)}
    \d{x}
    \leq
    \O \;
    \tv \! \left(
      \zeta;
      \mathopen]\xi-2\hat\lambda\theta, \xi\mathclose[
      \cup \mathopen]\xi, \xi+2\hat\lambda\theta\mathclose[
    \right) \,.
  \end{displaymath}
\end{lemma}

\begin{proof}
  Let $\hat u^\epsilon$ be an $\epsilon$-wave front tracking
  approximation of $S^{\zeta} u$ so that
  \begin{displaymath}
    \int_{\xi - \hat\lambda \theta}^{\xi + \hat\lambda \theta}
    \norma{S^\zeta_\theta u (x)   - S^{\tilde\zeta}_\theta u (x)}
    \d{x}
    \leq
    \O \, \epsilon
    +
    \int_{\xi - \hat\lambda \theta}^{\xi + \hat\lambda \theta}
    \norma{u^\epsilon (\theta,x)   -
      \left(S^{\tilde\zeta}_\theta u^\epsilon (0)\right)(x)}
    \d{x} \,.
  \end{displaymath}
  By~\cite[Theorem~2.9]{Bressan2000},
  \begin{eqnarray}
    \nonumber
    &
    & \int_{\xi - \hat\lambda \theta}^{\xi + \hat\lambda \theta}
      \norma{u^\epsilon (\theta,x)   -
      \left(S^{\tilde\zeta}_\theta u^\epsilon (0)\right)(x)}
      \d{x}
    \\
    \nonumber
    & \leq
    & \O \int_0^\theta
      \liminf_{h\to 0+} \dfrac1h
      \int_{\xi-2\hat\lambda\theta+\hat\lambda(t+h)}%
      ^{\xi+2\hat\lambda\theta-\hat\lambda (t+h)}
      \norma{u^\epsilon (t+h,x) - %
      \left(S^{\tilde\zeta}_h u^\epsilon (t)\right) (x)} \d{x} \d{t}
    \\
    \label{eq:74}
    & \leq
    & \O \int_0^\theta
      \liminf_{h\to 0+} \dfrac1h
      \int_{\xi-2\hat\lambda\theta+\hat\lambda(t+h)}%
      ^{\xi+2\hat\lambda\theta-\hat\lambda (t+h)}
      \norma{u^\epsilon (t+h,x) - %
      \left(S^{\zeta}_h u^\epsilon (t)\right) (x)} \d{x} \d{t}
    \\
    \label{eq:75}
    &
    & +\O \int_0^\theta
      \liminf_{h\to 0+} \dfrac1h
      \int_{\xi-2\hat\lambda\theta+\hat\lambda(t+h)}%
      ^{\xi+2\hat\lambda\theta-\hat\lambda (t+h)}
      \norma{\left(S^{\zeta}_h u^\epsilon (t)\right) (x) - %
      \left(S^{\tilde\zeta}_h u^\epsilon (t)\right) (x)} \d{x} \d{t} \,.
  \end{eqnarray}
  The integral in~\eqref{eq:74} is bounded by $\O\,\epsilon$ since
  $u^\epsilon$ is a piecewise constant $\epsilon$--approximation of
  the trajectory $t \mapsto S^{\zeta}_t \left(u^\epsilon
    (0)\right)$. The map $x \mapsto u^\epsilon (t,x)$ is piecewise
  constant, hence the integral in~\eqref{eq:75} can be computed
  estimating the differences in the local solutions to Riemann
  Problems arising from the discontinuities in $u^\epsilon (t)$ using
  Lemma~\ref{lem:zetaPbR} in the case $\check u_o = \hat u_o$. Thus,
  the term in~\eqref{eq:75} is estimated as
  \begin{eqnarray*}
    &
    & \liminf_{h\to 0+} \dfrac1h
      \int_{\xi-2\hat\lambda\theta+\hat\lambda(t+h)}%
      ^{\xi+2\hat\lambda\theta-\hat\lambda (t+h)}
      \norma{\left(S^{\zeta}_h u^\epsilon (t)\right) (x) - %
      \left(S^{\tilde\zeta}_h u^\epsilon (t)\right) (x)} \d{x}
    \\
    & \leq
    & \sum_{
      \substack{
      x_\alpha \in \mathcal{I} (u^\epsilon (t))
    \\
    x_\alpha \in
    ]{\xi-2\hat\lambda\theta+\hat\lambda t},
    {\xi+2\hat\lambda\theta-\hat\lambda t}[
    \\
    x_\alpha \neq\xi \,,\;
    x_\alpha \not\in \mathcal{I} (\zeta)
    }
    }
    \liminf_{h\to 0+} \dfrac1h
    \int_{x_\alpha-\hat\lambda h}%
    ^{x_\alpha + \hat\lambda h}
    \norma{\left(S^{\zeta}_h u^\epsilon (t)\right) (x) - %
    \left(S^{\tilde\zeta}_h u^\epsilon (t)\right) (x)} \d{x}
    \\
    &
    & + \sum_{
      \substack{
      x_\alpha \in \mathcal{I} (\zeta)
    \\
    x_\alpha \in
    ]{\xi-2\hat\lambda\theta+\hat\lambda t},
    {\xi+2\hat\lambda\theta-\hat\lambda t}[
    \\
    x_\alpha \neq\xi
    }
    }
    \liminf_{h\to 0+} \dfrac1h
    \int_{x_\alpha-\hat\lambda h}%
    ^{x_\alpha + \hat\lambda h}
    \norma{\left(S^{\zeta}_h u^\epsilon (t)\right) (x) - %
    \left(S^{\tilde\zeta}_h u^\epsilon (t)\right) (x)} \d{x}
    \\
    &
    & +
      \liminf_{h\to 0+} \dfrac1h
      \int_{\xi-\hat\lambda h}%
      ^{\xi + \hat\lambda h}
      \underbrace{\norma{\left(S^{\zeta}_h u^\epsilon (t)\right) (x) - %
      \left(S^{\tilde\zeta}_h u^\epsilon (t)\right) (x)}}_{=0} \d{x}
    \\
    & \leq
    & \O \sum_{
      \substack{
      x_\alpha \in \mathcal{I} (u^\epsilon (t))
    \\
    x_\alpha \in
    ]{\xi-2\hat\lambda\theta+\hat\lambda t},
    {\xi+2\hat\lambda\theta-\hat\lambda t}[
    \\
    x_\alpha \neq\xi \,,\;
    x_\alpha \not\in \mathcal{I} (\zeta)
    }
    }
    \norma{\Delta u^\epsilon (t,x_\alpha)} \,
    \norma{\zeta (x_\alpha) - \tilde\zeta (x_\alpha)}
    \\
    &
    & + \O \sum_{
      \substack{
      x_\alpha \in \mathcal{I} (\zeta)
    \\
    x_\alpha \in
    ]{\xi-2\hat\lambda\theta+\hat\lambda t},
    {\xi+2\hat\lambda\theta-\hat\lambda t}[
    \\
    x_\alpha \neq\xi
    }
    }
    \left(
    \norma{\Delta u^\epsilon (t,x_\alpha)} \,
    \norma{\zeta (x_\alpha) - \tilde\zeta (x_\alpha)}
    +
    \norma{\Delta\zeta (x_\alpha)}
    \right)
    \\
    & \leq
    & \O \, \tv \left(\zeta;\mathopen]\xi-2\hat\lambda\theta, \xi\mathclose[
      \cup \mathopen]\xi, \xi+2\hat\lambda\theta\mathclose[\right) \,,
  \end{eqnarray*}
  and, in the limit $\epsilon\to 0$, the proof of Lemma~\ref{lem:1}
  follows.
\end{proof}

\begin{lemma}
  \label{lem:2}
  Fix $u_o \in \mathcal{D}^\zeta$. For a $\xi \in \reali$ define
  \begin{equation}
    \label{eq:46}
    \tilde u (x)
    =
    \left\{
      \begin{array}{l@{\qquad}r@{\;}c@{\;}l}
        u_o (\xi-)
        & x
        & \in
        & \mathopen]\xi-\delta, \xi\mathclose[
        \\
        u_o (\xi+)
        & x
        & \in
        & \mathopen]\xi, \xi+\delta\mathclose[
        \\
        u_o (x)
        & x
        & \in
        & \mathopen]-\infty, \xi-\delta\mathclose[
          \cup \mathopen]\xi+\delta, +\infty\mathclose[
      \end{array}
    \right.
  \end{equation}
  and assume that $\tilde u \in \mathcal{D}^\zeta$. Then, for
  $\theta \in \mathopen]0, \delta/ (2\hat\lambda)\mathclose[$,
  \begin{displaymath}
    \dfrac1\theta
    \int_{\xi - \hat\lambda\theta}^{\xi+\hat\lambda\theta}
    \norma{S^\zeta_\theta u_o (x) - S^\zeta_\theta \tilde u (x)} \d{x}
    \leq
    \O \; \tv\left(
      u_o;
      \mathopen]\xi-2\hat\lambda\theta, \xi\mathclose[
      \cup \mathopen]\xi, \xi+2\hat\lambda\theta\mathclose[
    \right) \,.
  \end{displaymath}
\end{lemma}

\begin{proof}
  Use the Lipschitz continuity of $S^\zeta$, see Lemma~\ref{lem:123},
  on the dependency domain,
  \begin{eqnarray*}
    \dfrac1\theta
    \int_{\xi - \hat\lambda\theta}^{\xi+\hat\lambda\theta}
    \norma{S^\zeta_\theta u_o (x) - S^\zeta_\theta \tilde u (x)} \d{x}
    & \leq
    & \dfrac{L}{\theta}
      \int_{\xi - 2\hat\lambda\theta}^{\xi+2\hat\lambda\theta}
      \norma{u_o (x) - \tilde u (x)} \d{x}
    \\
    & \leq
    & \dfrac{L}{\theta}
      \int_{\xi - 2\hat\lambda\theta}^{\xi+2\hat\lambda\theta}
      \tv\left(
      u_o;
      \mathopen]\xi-2\hat\lambda\theta, \xi\mathclose[
      \cup \mathopen]\xi, \xi+2\hat\lambda\theta\mathclose[
      \right)
      \d{x}
  \end{eqnarray*} and the proof follows.
\end{proof}

We are now ready to complete the proof of~\emph{(i)} in
Theorem~\ref{thm:sgrp} for a piecewise constant $\zeta$. Use
$\tilde u$ as defined in~\eqref{eq:46} with $u_o = u (\tau)$ and
$\tilde\zeta$ as in~\eqref{eq:47}, so that
$U^\sharp_{(u,\tau,\xi)} (\theta,x) = (S^{\tilde\zeta}_\theta \tilde
u) (x)$ for $\theta$ in a right neighborhood of $0$ and $x$ near
$\xi$:
\begin{eqnarray*}
  &
  & \int_{\xi-\hat\lambda\theta}^{\xi+\hat\lambda\theta}
    \norma{u (\tau+\theta,x) - U^\sharp_{(u,\tau,\xi)} (\theta,x)}
    \d{x}
  \\   & =
  & \int_{\xi-\hat\lambda\theta}^{\xi+\hat\lambda\theta}
    \norma{
    (S^\zeta_\theta u (\tau)) (x)
    -
    (S^{\tilde\zeta}_\theta \tilde u (\tau)) (x)}
    \d{x}
  \\
  & \leq
  & \int_{\xi-\hat\lambda\theta}^{\xi+\hat\lambda\theta}
    \norma{
    (S^\zeta_\theta u (\tau)) (x)
    -
    (S^{\zeta}_\theta \tilde u (\tau)) (x)}
    \d{x}
    + \int_{\xi-\hat\lambda\theta}^{\xi+\hat\lambda\theta}
    \norma{
    (S^\zeta_\theta \tilde u (\tau)) (x)
    -
    (S^{\tilde\zeta}_\theta \tilde u (\tau)) (x)}
    \d{x}
  \\
\end{eqnarray*}
and the latter two terms are estimated by means of Lemma~\ref{lem:2}
and Lemma~\ref{lem:1}, obtaining
\begin{equation}
  \label{eq:38}
  \begin{array}{@{}cl@{}}
    & \displaystyle
      \dfrac1\theta
      \int_{\xi-\hat\lambda\theta}^{\xi+\hat\lambda\theta}
      \norma{u (\tau+\theta,x) - U^\sharp_{(u,\tau,\xi)} (\theta,x)}
      \d{x}
    \\
    \leq
    & \displaystyle
      \O
      \left(
      \tv\left(
      u,
      \mathopen]\xi-2\hat\lambda\theta, \xi\mathclose[
      \cup \mathopen]\xi, \xi+2\hat\lambda\theta\mathclose[
      \right)
      +
      \tv\left(
      \zeta,
      \mathopen]\xi-2\hat\lambda\theta, \xi\mathclose[
      \cup \mathopen]\xi, \xi+2\hat\lambda\theta\mathclose[
      \right)
      \right)
  \end{array}
\end{equation}
and the statement follows passing to the limit $\theta \to 0$.

We now head towards the proof of~\emph{(ii)} in
Theorem~\ref{thm:sgrp}. Preliminary is the following result.

\begin{lemma}
  \label{lem:LinSol}
  Let $A$ be an $n\times n$ non singular matrix with $n$ real
  eigenvalues $\lambda_1, \ldots, \lambda_n$, $n$ linearly independent
  right, respectively left, eigenvectors $r_1, \ldots, r_n$,
  respectively $l_1, \ldots, l_n$, and let $m$ be a finite vector
  measure. Then, the equation
  \begin{displaymath}
    \partial_t u + A \; \partial_x u
    =
    m
  \end{displaymath}
  generates the $\L1$--Lipschitz semigroup
  \begin{displaymath}
    \begin{array}{@{}c@{\,}c@{\,}ccc@{}}
      L_t
      & \colon
      & \L1 (\reali; \reali^n)
      & \to
      & \L1 (\reali; \reali^n)
      \\
      &
      & u
      & \to
      & \displaystyle
        \sum_{i=1}^n
        l_i \cdot
        \left(
        u (x-\lambda_i\, t)
        +
        \frac{1}{\lambda_i} \,
        \int_{x-\lambda_i t}^x \d{m}
        \right) r_i \,.
    \end{array}
  \end{displaymath}
\end{lemma}

\noindent The proof relies on a direct computation and is
omitted. Note for later use that
\begin{displaymath}
  L_t u
  =
  \sum_{i=1}^n
  l_i \cdot
  u (x-\lambda_i\, t)
  \; r_i
  +
  A^{-1}
  \sum_{i=1}^n
  l_i \cdot
  \int_{x-\lambda_i t}^x \d{m}
  \; r_i
\end{displaymath}

The next Lemma proves~\emph{(ii)} in Theorem~\ref{thm:sgrp} in the
case $x \mapsto u (\tau,x)$ is piecewise constant.

\begin{lemma}
  \label{lem:PCu}
  Under the same assumptions of Theorem~\ref{thm:sgrp}, assume
  moreover that $x \to u (\tau,x)$ is piecewise constant. Then,
  \emph{(ii)} in Theorem~\ref{thm:sgrp} holds.
\end{lemma}

\begin{proof}
  With the notation in Lemma~\ref{lem:LinSol}, recalling the
  definition~\eqref{eq:HelpC} of $U^\flat$ in the case of a piecewise
  constant $\zeta$,
  \begin{equation}
    \label{eq:37}
    U^\flat_{(u;\tau,\xi)} (\theta,\cdot) = L_\theta u (\tau)
    \mbox{ if }
    \left\{
      \begin{array}{@{}r@{\,}c@{\,}l@{}}
        A
        & =
        & D_u f \left(\zeta (\xi), u (\tau,\xi)\right)
        \\
        k (\bar x)
        & =
        & \Xi\left(\zeta (\bar x+), \zeta (\bar x), u (\tau, \xi)\right)
        \\
        &
        & \quad -
          f\left(\zeta (\bar x+),u (\tau, \xi)\right)
          +
          f\left(\zeta (\bar x),u (\tau, \xi)\right)
        \\
        m
        &  =
        &  \sum\limits_{\bar x\in \mathcal{I} (\zeta)}
          k (\bar x) \, \delta_{\bar x} \,.
      \end{array}
    \right.
  \end{equation}

  Below, set for simplicity $u_o = u (\tau)$ and assume that
  $\tau=0$. Use~\cite[Theorem~2.9]{Bressan2000} with the notation in
  the statement of Theorem~\ref{thm:sgrp}, recalling that
  $x \mapsto L_t u_o (x)$ is piecewise constant, since so is $u_o$ and
  $L$ is linear.
  \begin{eqnarray*}
    &
    & \dfrac{1}{\theta}
      \int_{a+\hat\lambda \theta}^{b-\hat\lambda \theta}
      \norma{S_\theta u_o - U^\flat_{(u_o;0,\xi)}} \d{x}
    \\
    & =
    & \dfrac{1}{\theta}
      \int_{a+\hat\lambda \theta}^{b-\hat\lambda \theta}
      \norma{S_\theta u_o - L_\theta u_o} \d{x}
    \\
    & \leq
    & \O \, \dfrac{1}{\theta} \int_0^\theta
      \liminf_{h \to 0+}
      \dfrac{1}{h}
      \int_{a + \hat\lambda (t+h)}^{b - \hat\lambda (t+h)}
      \norma{S_h L_t u_o - L_h L_t u_o} \d{x} \d{t}
    \\
    & \leq
    & \O \, \dfrac{1}{\theta} \int_0^\theta
      \sum_{\bar x \in \mathcal{I} (L_t u_o) \cup \mathcal{I} (\zeta)}
      \liminf_{h \to 0+}
      \dfrac{1}{h}
      \int_{\bar x-\hat\lambda h}^{\bar x +\hat\lambda h}
      \norma{S_h L_t u_o - L_h L_t u_o} \d{x} \d{t} \,.
  \end{eqnarray*}
  To compute the latter sum, we distinguish the two cases
  $\bar x \not\in \mathcal{I} (\zeta)$ or
  $\bar x \in \mathcal{I} (\zeta)$.

  In the former case, in a neighborhood of $\bar x$ we have that the
  semigroup $L$ locally coincide with that generated by the
  homogeneous equation $\partial_t u + A \, \partial_x u = 0$. Hence,
  by~\cite[Formula~(3.8)]{BianchiniColombo2002}, Lemma~\ref{lem:PCu}
  and~\eqref{eq:37}, we have
  \begin{eqnarray*}
    &
    & \dfrac{1}{h}
      \int_{\bar x-\hat\lambda h}^{\bar x +\hat\lambda h}
      \norma{S_h L_t u_o - L_h L_t u_o} \d{x}
    \\
    & \leq
    &  \O
      \sup_{x \in\mathopen]a,b\mathclose[} \norma{
      A - D_u f\left(\zeta (x), (L_t u_o) (x)\right)
      }
      \norma{\Delta (L_t \, u_o) (\bar x)}
    \\
    & \leq
    & \O
      \sup_{x \in\mathopen]a,b\mathclose[}
      \norma{
      D_u f\left(\zeta (\xi),u_o(\xi)\right) -
      D_u f\left(\zeta (x), (L_t u_o) (x)\right)
      }
      \norma{\Delta (L_t \, u_o) (\bar x)}
    \\
    & \leq
    & \O
      \sup_{x \in\mathopen]a,b\mathclose[}
      \left(
      \norma{\zeta (x) - \zeta (\xi)}
      +
      \norma{u_o(\xi) - (L_t u_o) (x)}
      \right)
      \norma{\Delta (L_t \, u_o) (\bar x)}
    \\
    & \leq
    & \O \, \norma{\Delta (L_t u_o) (\bar x)}
      \left(
      \tv\left(u_o, \mathopen]a,b\mathclose[\right)
      +
      \tv\left(\zeta, \mathopen]a,b\mathclose[\right)
      \right) \,.
  \end{eqnarray*}
  Hence,
  \begin{eqnarray*}
    &
    &  \sum_{\bar x \in \mathcal{I} (L_t u_o) , \bar x \not \in \mathcal{I}(\zeta)}
      \dfrac{1}{h}
      \int_{\bar x-\hat\lambda h}^{\bar x +\hat\lambda h}
      \norma{S_h L_t u_o - L_h L_t u_o} \d{x}
    \\
    & \leq
    &  \O \sum_{\bar x \in \mathcal{I} (L_t u_o) , \bar x \not \in \mathcal{I}(\zeta)}
      \norma{\Delta (L_t u_o) (\bar x)}
      \left(
      \tv\left(u_o, \mathopen]a,b\mathclose[\right)
      +
      \tv\left(\zeta, \mathopen]a,b\mathclose[\right)
      \right)
    \\
    & \leq
    & \O
      \left(
      \tv\left(u_o, \mathopen]a,b\mathclose[\right)
      +
      \tv\left(\zeta, \mathopen]a,b\mathclose[\right)
      \right)^2 \,.
  \end{eqnarray*}

  Assume now that $\bar x \in \mathcal{I} (\zeta)$. Using the map $T$
  defined in Lemma~\ref{lem:ExistenceT}, define
  \begin{eqnarray*}
    \tilde u
    & =
    & (L_tu_o) (\bar x-)
    \\
    w (x)
    & =
    & \left\{
      \begin{array}{l@{\quad}r@{\,}c@{\,}l}
        \tilde u
        & x
        & <
        & \bar x
        \\
        T\left(\zeta (\bar x+), \zeta (\bar x-)\right)
        (\tilde u)
        & x
        & >
        & \bar x
      \end{array}
          \right.
  \end{eqnarray*}
  and note that
  $(L_t u_o) (\bar x+) = \tilde u + A^{-1}\, k (\bar x)$.  Then,
  $S_h w = w$ and $L_h L_t u_o = L_t u_o$ in a neighborhood of
  $\bar x$. Moreover, recall the Lipschitz continuity of $S_h$
  restricted to dependency domains:
  \begin{displaymath}
    \int_{\bar x - \bar\lambda h}^{\bar x + \bar\lambda h}
    \norma{S_h L_t u_o - S_h w} \d{x}
    \leq
    \int_{\bar x - 2\bar\lambda h}^{\bar x + 2\bar\lambda h}
    \norma{L_t u_o - w} \d{x} \,,
  \end{displaymath}
  and, using the notation in~\eqref{eq:37}, proceed
  \begin{eqnarray*}
    &
    & \dfrac{1}{h}
      \int_{\bar x - \hat\lambda h}^{\bar x +\hat\lambda h}
      \norma{S_h L_t u_o - L_h L_t u_o} \d{x}
    \\
    & \leq
    & \dfrac{1}{h}
      \int_{\bar x - \hat\lambda h}^{\bar x +\hat\lambda h}
      \norma{S_h L_t u_o - S_h w} \d{x}
      +
      \dfrac{1}{h}
      \int_{\bar x - \hat\lambda h}^{\bar x +\hat\lambda h}
      \norma{S_h w - L_h L_t u_o} \d{x}
    \\
    & \leq
    & \O \,
      \norma{
      T\left(\zeta (\bar x+), \zeta (\bar x)\right)
      (\tilde u)
      -
      \left(
      \tilde u
      +
      A^{-1} \, k (\bar x)
      \right)}
    \\
    & \leq
    & \O \,
      \norma{
      A \, T\left(\zeta (\bar x+), \zeta (\bar x)\right)
      (\tilde u)
      -
      \left(
      A \tilde u
      +
      k (\bar x)
      \right)}
    \\
    & =
    & \O \, \left\|
      D_u f\left(\zeta (\xi), u_o (\xi)\right)
      \left(
      T\left(\zeta (\bar x+), \zeta (\bar x)\right)
      (\tilde u)
      -
      \tilde u\right)
      \right.
    \\
    &
    & \left.
      \qquad
      -\Xi\left(\zeta (\bar x+), \zeta (\bar x), u_o (\xi)\right)
      +
      f\left(\zeta (\bar x+),u_o (\xi)\right)
      -
      f\left(\zeta (\bar x),u_o (\xi)\right)
      \right\|
    \\
    & \leq
    & \O \,
      \norma{\Delta \zeta (\bar x)}
      \left(
      \norma{\zeta (\bar x+) - \zeta (\xi)}
      +
      \norma{\Delta \zeta (\bar x)}
      +
      \norma{u_o (\xi) - \tilde u}
      \right)
    \\
    & \leq
    & \O \,
      \norma{\Delta \zeta (\bar x)}
      \left(
      \tv\left(u_o, \mathopen]a,b\mathclose[\right)
      +
      \tv\left(\zeta, \mathopen]a,b\mathclose[\right)
      \right)
      \,,
  \end{eqnarray*}
  where we used~2.~in Lemma~\ref{lem:ExistenceT}. Now, we add over
  $\bar x \in \mathcal{I} (\zeta)$:
  \begin{displaymath}
    \sum_{x \in \mathcal{I} (\zeta)}
    \dfrac{1}{h}
    \int_{\bar x - \hat\lambda h}^{\bar x +\hat\lambda h}
    \norma{S_h L_t u_o - L_h L_t u_o} \d{x}
    \leq
    \O
    \left(
      \tv\left(u_o, \mathopen]a,b\mathclose[\right)
      +
      \tv\left(\zeta, \mathopen]a,b\mathclose[\right)
    \right)^2
    \,,
  \end{displaymath}
  completing the proof of Lemma~\ref{lem:PCu}.
\end{proof}

To complete the proof of~\emph{(ii)} in Theorem~\ref{thm:sgrp},
consider the case of $x \to u (\tau,x)$ not necessarily piecewise
constant. We keep using the equalities
$u (\tau+\theta) = S_\theta\left(u (\tau)\right)$ and
$U^\flat_{(u;\tau,\xi)} = L_\theta\left(u (\tau)\right)$, the linear
operator $L$ being defined in Lemma~\ref{lem:LinSol} with $A$ and $k$
as in~\eqref{eq:37}. Then, for $\epsilon>0$ call $u^\epsilon$ a
piecewise constant approximation of $u (\tau)$ with
$\tv (u^\epsilon) \leq \tv \left(u (\tau)\right)$. By the Lipschitz
continuity of $S_\theta$ and $L_t$, we have
\begin{eqnarray*}
  &
  & \dfrac1\theta \int_{a-\hat\lambda \theta}^{b+\hat\lambda \theta}
    \norma{u (\tau+\theta,x) - U^\flat_{(u;\tau,\xi)} (\theta,x)} \d{x}
  \\
  & \leq
  & \O \dfrac{\epsilon}\theta
    + \dfrac1\theta
    \int_{a-\hat\lambda \theta}^{b+\hat\lambda \theta}
    \norma{(S_\theta u^\epsilon) (x) - (L_\theta u^\epsilon) (x)} \d{x}
  \\
  & \leq
  & \O \dfrac{\epsilon}\theta
    +
    \O
    \left(
    \tv\left(u (\tau), \mathopen]a,b\mathclose[\right)
    +
    \tv\left(\zeta, \mathopen]a,b\mathclose[\right)
    \right)^2 \,.
\end{eqnarray*}
Passing first to the limit $\epsilon \to 0$ and then to the limit
$\theta \to 0$, we complete the proof.

\subsection{The General Case}

Consider now the case $\zeta \in \BV (\reali; \reali^p)$.

\begin{lemma}
  \label{lem:uno}
  Assume $f$ satisfies~\ref{it:f1}--\ref{it:f4} and $\Xi$
  satisfies~\ref{eq:Xi1}--\ref{eq:Xi4}. Then, for all
  $z \in \mathcal{Z}$, $v \in \reali^p$.  $\omega \in \Omega$, if
  $\delta>0$ is sufficiently small,
  \begin{displaymath}
    \Xi (z+\delta\, v, z, \omega)
    -
    f (z+\delta\, v, \omega)
    +
    f (z,\omega)
    =
    \delta
    \left(
      D_v^+\Xi (z,z,\omega)
      -
      D_zf (z,\omega) \, v
    \right)
    + \O \, \sigma (\delta) \, \delta \,.
  \end{displaymath}
\end{lemma}

\noindent The proof directly follows from--\ref{eq:Xi4} and from the
Taylor expansion of $f$.

\begin{proofof}{Theorem~\ref{thm:sgrp}}
  Let $\zeta \in \BV (\reali; \mathcal{Z})$. Call
  $\mathcal{I} (\zeta)$ the, at most countable, set of points of jump
  in $\zeta$. Recall that $D\zeta$ is a finite vector measure. Let
  $\mu$ and $v$ be as in~\eqref{eq:20}. By Lusin
  Theorem~\cite[Theorem~2.24]{Rudin}, for any $h > 0$, there exists a
  $\tilde v^h \in \Cc0 (\reali; \reali^p)$ such that
  \begin{equation}
    \label{eq:29}
    \norma{\tilde v^h (x)} \leq 1
    \quad \mbox{ and } \quad
    \norma{D\zeta}
    \left(
      \left\{x \in \reali \colon \tilde v^h (x) \neq v (x)\right\}
    \right) < h \,.
  \end{equation}
  Following~\cite[Step~1, \S~4.3]{ColomboGuerraHolle}, introduce
  points $\{ x_1, \ldots, x_{N_h-1}\} \in \reali$ such that:
  \begin{enumerate}[label={\bf{(\roman*)}}]
  \item \label{item:r1} $x_0 = -\infty$, $x_1 < -1/h$, $x_{i-1} < x_i$
    for $i = 2, \ldots, N_h-1$, $x_{N_h-1} > 1/h$ and
    $x_{N_h} = +\infty$.
  \item \label{item:r2}
    $\sum_{x \in \mathcal{I} (\zeta)\setminus \mathcal{I}^h}
    \left\|\Delta \zeta(x)\right\|< h$ for a suitable set of points
    $\mathcal{I}^h$ contained in $\{x_1, x_2, \ldots, x_{N_h -1}\}$.
  \item \label{item:r3} Whenever\footnote{Everywhere, $\sharp A$
      stands the (finite) cardinality of the set $A$.}
    $x_i \in \mathcal{I}^h$,
    $\tv\left(\zeta, \left[x_{i-1}, x_i \right[\right) < \left. h
      \middle/ (1+\sharp \mathcal{I}^h) \right.$.
  \item \label{item:r4}
    $\tv \left(\zeta, \left]x_{i-1}, x_i\right[\right) < h$ for all
    $i = 1, \ldots, N_h$.
  \item \label{item:r6}
    $\norma{\tilde v^h \left(x'\right) - \tilde v^h \left(x''\right)}
    < h$ for $x',\,x''\in \mathopen]x_{i-1},x_{i}\mathclose[$,
    $i=1,\ldots,N_{h}$.

  \item \label{item:r7} $x_i - x_{i-1} \in \left]0, h\right[$ for all
    $i=2, \ldots, N_h-1$.
  \end{enumerate}
  \noindent Points satisfying~\ref{item:r1} are easily
  constructed. Then, adding more points, one fulfills
  also~\ref{item:r2} and this condition fully defines $\mathcal{I}^h$
  and, hence, $\sharp \mathcal{I}^h$. Iteratively continuing to add
  points, thus increasing $N_h$, we satisfy also~\ref{item:r3},
  \ref{item:r4}, \ref{item:r6} and~\ref{item:r7}, in this
  order. Define the piecewise constant map
  \begin{equation}
    \label{eq:24}
    \zeta^h
    =
    \zeta\left(-\infty\right)\caratt{]-\infty,x_{1}]}
    +
    \sum_{i=2}^{N_h-1} \zeta \left(x_{i-1}+\right) \;
    \caratt{]x_{i-1}, x_i]}
    +
    \zeta\left(x_{N_{h}-1}+\right)\caratt{]x_{N_{h}-1},+\infty [}
  \end{equation}
  and note that
  \begin{equation}
    \label{eq:51}
    \zeta^h (x_i) = \zeta (x_{i-1}+)
    \quad \mbox{ and } \quad
    \zeta^h (x_i+) = \zeta (x_i+) \,.
  \end{equation}
  The approximations $\zeta^h$ converge to $\zeta$ uniformly on
  $\reali$ as $h \to 0$. Indeed, fix $h$ and for any $x \in \reali$,
  by~\ref{item:r1} we have $x \in \left]x_{i-1}, x_i\right]$
  (obviously excluding $+\infty$) and for $i \in \{1, \ldots, N_h\}$,
  by~\ref{item:r4},
  \begin{displaymath}
    \norma{\zeta^h (x) - \zeta (x)}
    \leq
    \tv\left(\zeta, \left]x_{i-1}, x_i\right[\right)
    \leq h \,.
  \end{displaymath}
  Call
  $S^{\zeta^h} \colon \mathopen[0, +\infty\mathclose[ \times
  \mathcal{D}_{\delta}^{\zeta^h} \to \mathcal{D}_\delta^{\zeta^h}$ the
  semigroup whose existence is proved in the piecewise constant case
  in \S~\ref{subs:proof--theor-refthm:s}, provided $\delta$ is
  sufficiently small. We prove that as $h\to 0$ the semigroups
  $S^{\zeta^h}$ converge to a semigroup $S^\zeta$ in $\L1$.

  Using the notation~\eqref{def:2.6}, introduce the sets:
  \begin{equation}
    \label{eq:22}
    \displaystyle
    \check{\mathcal{D}}_\delta^\zeta
    =
    \bigcap_{h>0} \mathcal{D}_\delta^{\zeta^h}
    \qquad\qquad\qquad
    \hat{\mathcal{D}}_\delta^\zeta
    =
    \overline{\bigcup_{h>0} \mathcal{D}_\delta^{\zeta^h}}
  \end{equation}
  the latter closure is understood in the strong $\L1$ topology. If
  $\zeta$ has sufficiently small total variation then suitably
  choosing positive $\delta$ and $\delta'$
  \begin{displaymath}
    \check{\mathcal{D}}_\delta^\zeta
    \subseteq
    \hat{\mathcal{D}}_\delta^\zeta
    \subseteq
    \check{\mathcal{D}}_{\delta'}^\zeta
  \end{displaymath}
  and all these sets are not empty since they contain all $u$ with
  sufficiently small total variation.

  Since $\check{\mathcal{D}}_{\delta'}^{\zeta}$ is separable with
  respect to the strong $\L1$ topology, by a diagonalization process
  there exists a sequence $h_i$ such that for all
  $u \in \check{\mathcal{D}}_{\delta'}^\zeta$ and for all
  $t \in \mathopen[0, +\infty\mathclose[$, the sequence
  $S^{\zeta^{h_i}}_t u$ converges in $\Lloc1 (\reali; \reali^n)$ to a
  limit which we define as $S^\zeta_t u$. Clearly,
  $S_t^\zeta u \in \hat{\mathcal{D}}_{\delta'}^\zeta$. Moreover,
  whenever $S_t^\zeta u \in \check{\mathcal{D}}_{\delta'}^\zeta$,
  thanks to the Lipschitz continuity of
  $u \mapsto S^{\zeta_{h_i}}_t u$, the semigroup property holds in the
  limit $h_i \to 0$, i.e., $S_s^\zeta S_t^\zeta u = S_{s+t}^\zeta u$
  for all $s \geq 0$.

  Define now
  \begin{equation}
    \label{eq:27}
    \displaystyle
    \mathcal{D}_\delta^\zeta
    =
    \left\{
      u \in \hat{\mathcal{D}}_\delta^\zeta \colon
      \exists\,t \in \reali_+ \mbox{ and } \exists w \in
      \check{\mathcal{D}}_\delta^\zeta \mbox{ such that } S^\zeta_t w = u
    \right\} \,.
  \end{equation}
  Note that
  \begin{displaymath}
    \check{\mathcal{D}}_\delta^\zeta
    \subseteq
    \mathcal{D}_\delta^\zeta
    \subseteq
    \hat{\mathcal{D}}_\delta^\zeta \,.
  \end{displaymath}
  For all $t \in \reali_+$, the domain $\mathcal{D}_\delta^\zeta$ is
  invariant with respect to $S_t^\zeta$, in the sense that
  $(S_t^\zeta \mathcal{D}_\delta^\zeta) \subseteq
  \mathcal{D}_\delta^\zeta$.

  Following the lines of~\cite[Theorem~2.2]{ColomboGuerraHolle}, the
  above construction proves 1.~in
  Theorem~\ref{thm:sgrp}. Condition~2.~in Theorem~\ref{thm:sgrp} also
  follows, since the semigroup $S^{\zeta^{h_i}}$ admits a Lipschitz
  constant independent of $h_i$. Statement~3.~in
  Theorem~\ref{thm:sgrp} now follows from the results in
  \S~\ref{subs:proof--theor-refthm:s}, since for $h$ sufficiently
  small $\zeta^h$ coincides with $\zeta$.

  To prove~4.~in Theorem~\ref{thm:sgrp}, we consider first~\emph{(i)}.

  \paragraph{Proof of~\textit{(i)}.} To simplify the notation, we
  denote $h_i$ by $h$. By $U^{\sharp h}_{(u;\tau,\xi)}$ denote the
  solution to the Riemann Problem~\eqref{eq:64} with $\zeta$ replaced
  by $\zeta^h$. Clearly,
  \begin{eqnarray}
    \label{eq:54}
    &
    & \dfrac1\theta \int_{\xi-\hat\lambda \theta}^{\xi + \hat\lambda \theta}
      \norma{
      \left(S_\theta^\zeta u (\tau)\right) (x)
      -
      U^\sharp_{(u;\tau,\xi)} (\theta,x)
      }
      \d{x}
    \\
    \nonumber
    & \leq
    & \dfrac1\theta
      \int_{\xi-\hat\lambda \theta}^{\xi + \hat\lambda \theta}
      \norma{
      \left(S_\theta^\zeta u (\tau)\right) (x)
      -
      \left(S_\theta^{\zeta^h} u (\tau)\right) (x)
      }
      \d{x}
    \\
    \nonumber
    &
    & +
      \dfrac1\theta \int_{\xi-\hat\lambda \theta}^{\xi + \hat\lambda \theta}
      \norma{
      \left(S_\theta^{\zeta^h} u (\tau)\right) (x)
      -
      U^{\sharp h}_{(u;\tau,\xi)} (\theta,x)
      }
      \d{x}
    \\
    \label{eq:76}
    &
    & +
      \dfrac1\theta \int_{\xi-\hat\lambda \theta}^{\xi + \hat\lambda \theta}
      \norma{
      U^{\sharp h}_{(u;\tau,\xi)} (\theta,x)
      -
      U^\sharp_{(u;\tau,\xi)} (\theta,x)
      }
      \d{x}.
  \end{eqnarray}
  The first integral in the right hand side above vanishes in the
  limit $h \to 0$. To estimate the second integral, use~\eqref{eq:38}
  and~\cite[Formula~(4.29)]{ColomboGuerraHolle} to get
  \begin{eqnarray*}
    &
    & \dfrac1\theta \int_{\xi-\hat\lambda \theta}^{\xi + \hat\lambda \theta}
      \norma{
      \left(S_\theta^{\zeta^h} u (\tau)\right) (x)
      -
      U^{\sharp h}_{(u;\tau,\xi)} (\theta,x)
      }
      \d{x}
    \\
    & \leq
    & \O
      \left(
      \tv\left(
      u (\tau),
      \mathopen]\xi-2\hat\lambda\theta, \xi\mathclose[
      \cup \mathopen]\xi, \xi+2\hat\lambda\theta\mathclose[
      \right)
      +
      \tv\left(
      \zeta^h,
      \mathopen]\xi-2\hat\lambda\theta, \xi\mathclose[
      \cup \mathopen]\xi, \xi+2\hat\lambda\theta\mathclose[
      \right)
      \right)
    \\
    & \leq
    & \O
      \left(
      \tv\left(
      u (\tau),
      \mathopen]\xi-2\hat\lambda\theta, \xi\mathclose[
      \cup \mathopen]\xi, \xi+2\hat\lambda\theta\mathclose[
      \right)
      {+}
      \tv\left(
      \zeta,
      \mathopen]\xi-2\hat\lambda\theta, \xi\mathclose[
      \cup \mathopen]\xi, \xi+2\hat\lambda\theta\mathclose[
      \right)
      +
      h
      \right) .
  \end{eqnarray*}
  In the limits, first for $h\to 0$ and then for $\theta \to 0$, the
  latter term above vanishes.

  Concerning the third term~\eqref{eq:76}, use Lemma~\ref{lem:zetaPbR}
  and obtain
  \begin{displaymath}
    \dfrac1\theta \int_{\xi-\hat\lambda \theta}^{\xi + \hat\lambda \theta}
    \norma{
      U^{\sharp h}_{(u;\tau,\xi)} (\theta,x)
      -
      U^\sharp_{(u;\tau,\xi)} (\theta,x)
    }
    \d{x}
    \leq
    \O
    \left(
      \norma{\zeta^h (\xi) - \zeta (\xi)}
      +
      \norma{\zeta^h (\xi+) - \zeta (\xi+)}
    \right)
  \end{displaymath}
  which vanishes as $h\to 0$ by the uniform convergence of $\zeta^h$
  to $\zeta$, completing the proof of~\emph{(i)}.

  \paragraph{Proof of~\textit{(ii)}.}
  We now pass to~\emph{(ii)} in item~4.~of Theorem~\ref{thm:sgrp}.
  the following definitions and preliminary results are of use below.

  Recall the notation in~\eqref{eq:20}. For $i=1, \ldots, N_h$, let
  \begin{eqnarray}
    \label{eq:60}
    \delta_i
    & =
    & \norma{\mu}\left(\mathopen]x_{i-1}, x_i \mathclose[\right) \,;
    \\
    \label{eq:62}
    v_i
    & =
    & \left\{
      \begin{array}{l@{\qquad}r@{\,}c@{\,}l}
        \frac{1}{\delta_i} \, \mu\left(\mathopen]x_{i-1}, x_i \mathclose[\right)
        & \delta_i
        & \neq
        & 0 \,;
        \\
        0
        & \delta_i
        & =
        & 0 \,;
      \end{array}
          \right.
    \\
    \label{eq:61}
    v^h
    & =
    & \sum_{i=1}^{N_h-1} v_i \; \caratt{\mathopen]x_{i-1}, x_i \mathclose]}
      + v_{\strut N_h} \;
      \caratt{\mathopen]x_{\strut N_h-1}, +\infty \mathclose[} \,.
  \end{eqnarray}
  Note also that for $\delta_i \neq 0$,
  $v_i = \frac{1}{\delta_i} \int_{\mathopen]x_{i-1}, x_i \mathclose[}
  v \, \d{\norma{\mu}}$.

  \paragraph{Claim: }We have the convergence
  \begin{equation}
    \label{eq:63}
    \lim_{h\to 0}
    \int_{\reali} \norma{v^h - v}
    \d{\norma{\mu}}
    = 0
    \,.
  \end{equation}
  Indeed, recalling~\eqref{eq:29},
  \begin{eqnarray}
    \nonumber
    &
    &
      \int_{\reali} \norma{v^h (x) - v (x)}
      \d{\norma{\mu} (x)}
    \\
    \nonumber
    & =
    & \sum_{i=1}^{N_h} \int_{]x_{i-1}, x_i[}
      \norma{v (x) - v_i}
      \d{\norma{\mu} (x)}
    \\
    \nonumber
    & =
    & \sum_{i=1, N_h \atop \delta_i \neq 0}
      \dfrac{1}{\delta_i}
      \int_{(\mathopen]x_{i-1},x_i\mathclose[)^2}
      \norma{v (x) - v (y)} \d{(\norma{\mu}\otimes\norma{\mu})}  (x,y)
    \\
    \label{eq:43}
    & =
    & \!\!\!\!\!\!
      \sum_{i=1, N_h\atop \delta_i \neq 0} \!
      \dfrac{1}{\delta_i} \,
      \int_{(\mathopen]x_{i-1},x_i\mathclose[)^2} \!
      \left[
      \norma{v (x) - \tilde v^h (x)}
      {+} \norma{\tilde v^h (y) - v (y)}
      \right]
      \d{(\norma{\mu}\otimes\norma{\mu})}  (x,y)
    \\
    \label{eq:44}
    &
    & + \sum_{i=1, N_h\atop \delta_i \neq 0}
      \dfrac{1}{\delta_i}
      \int_{(\mathopen]x_{i-1},x_i\mathclose[)^2}
      \norma{\tilde v^h (x) - \tilde v^h (y)}
      \d{(\norma{\mu}\otimes\norma{\mu})}  (x,y) \,.
  \end{eqnarray}
  The two terms in the integral in~\eqref{eq:43} are estimated in the
  same way, using~\eqref{eq:29}, as
  \begin{eqnarray*}
    &
    & \sum_{i=1, N_h\atop \delta_i \neq 0}
      \int_{(\mathopen]x_{i-1},x_i\mathclose[)^2}
      \dfrac{1}{\delta_i}
      \norma{v (x) - \tilde v^h (x)}
      \d{(\norma{\mu}\otimes\norma{\mu})} \! (x,y)
    \\
    & =
    & \sum_{i=1, N_h\atop \delta_i \neq 0}
      \int_{\mathopen]x_{i-1},x_i\mathclose[}
      \norma{v (x) - \tilde v^h (x)}
      \d{\norma{\mu}} \! (x)
    \\
    & =
    & \int_{\reali}
      \norma{v (x) - \tilde v^h (x)}
      \d{\norma{\mu}} \! (x)
    \\
    & \leq
    & \int_{\left\{x \in\reali\colon v (x) \neq \tilde v^h (x)\right\}}
      \left(\norma{v (x)} + \norma{\tilde v^h (x)}\right)
      \d{\norma{\mu}} \! (x)
    \\
    &\leq
    & 2 \, h
    \\
    & \to
    & 0 \quad \mbox{ as } h \to 0\,.
  \end{eqnarray*}
  We now estimate the term~\eqref{eq:44} by means of~\ref{item:r6}:
  \begin{eqnarray*}
    &
    & \sum_{i=1, N_h\atop \delta_i \neq 0}
      \dfrac{1}{\delta_i}
      \int_{(\mathopen]x_{i-1},x_i\mathclose[)^2}
      \norma{\tilde v^h (x) - \tilde v^h (y)}
      \d{(\norma{\mu}\otimes\norma{\mu})} \! (x,y)
    \\
    & \leq
    & h
      \sum_{i=1, N_h\atop \delta_i \neq 0}
      \dfrac{1}{\delta_i}
      \int_{(\mathopen]x_{i-1},x_i\mathclose[)^2}
      \d{(\norma{\mu}\otimes\norma{\mu})} \! (x,y)
    \\
    & \leq
    & h
      \sum_{i=1}^{N_h}
      \delta_i
    \\
    & \leq
    & h \; \norma{\mu} (\reali)
    \\
    & \to
    & 0 \quad \mbox{ as } h \to 0\,,
  \end{eqnarray*}
  completing the proof of the Claim.

  Apply Lemma~\ref{lem:LinSol} with
  $A = Df\left(\zeta (\xi), u (\tau,\xi)\right)$, first with $m = g$
  as defined in~\eqref{eq:57}, then with $m = g^h$ where $g^h$ is
  defined, for all Borel subset $E$ of $\reali$, by
  \begin{displaymath}
    g^h (E)
     =
    \sum_{y \in \mathcal{I}(\zeta^h)} \!
    \left(
      \Xi\left(\zeta^h (y+), \zeta^h (y), u (\tau, \xi)\right)
      -
      f\left(\zeta^h (y+),u (\tau, \xi)\right)
      +
      f\left(\zeta^h (y),u (\tau, \xi)\right)\right) \delta_{y} (E)
  \end{displaymath}
  and write
  \begin{equation}
    \label{eq:58}
    \begin{array}{rcl}
      U^\flat_{(u;\tau,\xi)} (\theta,x)
      & =
      & \displaystyle
        \sum_{i=1}^n
        l_i \cdot
        \left(
        u (\tau, x-\lambda_i\, \theta)
        +
        \frac{1}{\lambda_i} \,
        \int_{x-\lambda_i \theta}^x \d{g}
        \right) r_i,
      \\
      U^{\flat h}_{(u;\tau,\xi)} (\theta,x)
      & =
      & \displaystyle
        \sum_{i=1}^n
        l_i \cdot
        \left(
        u (\tau, x-\lambda_i\, \theta)
        +
        \frac{1}{\lambda_i} \,
        \int_{x-\lambda_i \theta}^x \d{g^h}
        \right) r_i.
    \end{array}
  \end{equation}
  Similarly to~\eqref{eq:54}, fix $a,b,\xi$ in $\reali$ with
  $a < \xi < b$, let
  $\theta \in \mathopen]0, (b-a)/\hat\lambda\mathclose[$ and compute
  \begin{eqnarray}
    \nonumber
    &
    & \dfrac1\theta \int_{a + \hat\lambda \theta}^{b - \hat\lambda \theta}
      \norma{
      \left(S_\theta^\zeta u (\tau)\right) (x)
      -
      U^\flat_{(u;\tau,\xi)} (\theta,x)
      }
      \d{x}
    \\
    \label{eq:50}
    & \leq
    & \dfrac1\theta
      \int_{a + \hat\lambda \theta}^{b - \hat\lambda \theta}
      \norma{
      \left(S_\theta^\zeta u (\tau)\right) (x)
      -
      \left(S_\theta^{\zeta^h} u (\tau)\right) (x)
      }
      \d{x}
    \\
    \label{eq:56}
    &
    & +
      \dfrac1\theta \int_{a + \hat\lambda \theta}^{b - \hat\lambda \theta}
      \norma{
      \left(S_\theta^{\zeta^h} u (\tau)\right) (x)
      -
      U^{\flat h}_{(u;\tau,\xi)} (\theta,x)
      }
      \d{x}
    \\
    \label{eq:55}
    &
    & +
      \dfrac1\theta \int_{a + \hat\lambda \theta}^{b - \hat\lambda \theta}
      \norma{
      U^{\flat h}_{(u;\tau,\xi)} (\theta,x)
      -
      U^\flat_{(u;\tau,\xi)} (\theta,x)
      }
      \d{x}.
  \end{eqnarray}
  The first term~\eqref{eq:50} vanishes as $h \to 0$ by the above
  construction of $S$.

  Since $\zeta^h$ is piecewise constant, to bound~\eqref{eq:56} we can
  use~(ii) in Theorem~\ref{thm:sgrp} as proved
  in~\S~\ref{subs:proof--theor-refthm:s} in the piecewise constant
  case:
  \begin{eqnarray*} [\mbox{\eqref{eq:56}}]
    & \leq
    & C\left[\tv\left(u(\tau), \mathopen]a,b\mathclose[\right)
      +
      \tv\left(\zeta^h, \mathopen]a,b\mathclose[\right)
      \right]^2
    \\
    & \leq
    & C\left[\tv\left(u(\tau), \mathopen]a,b\mathclose[\right)
      +
      h
      +
      \tv\left(\zeta, \mathopen]a,b\mathclose[\right)
      \right]^2
  \end{eqnarray*}
  where we used~\cite[Formula~(4.29)]{ColomboGuerraHolle}. In the
  limit $h \to 0$ we obtain the desired estimate.

  Compute~\eqref{eq:55} by means of~\eqref{eq:58} as
  \begin{eqnarray*} [\mbox{\eqref{eq:55}}]
    & =
    & \int_{a + \hat\lambda \theta}^{b - \hat\lambda \theta}
      \norma{
      \sum_{i=1}^n
      \dfrac{1}{\lambda_i}
      l_i \cdot
      \int_{x - \lambda_i \theta}^x
      \left(\d{g} - \d{g^h}\right)
      }
      \d{x}
    \\
    & \leq
    & \O \,
      \sum_{i=1}^n
      \int_{a + \hat\lambda \theta}^{b - \hat\lambda \theta}
      \norma{\int_{x - \lambda_i \theta}^x
      \left(\d{g} - \d{g^h}\right)
      }
      \d{x} \,.
  \end{eqnarray*}
  We now estimate the latter integrals, assuming that neither $x$ nor
  $x-\lambda_i \theta$ are discontinuity points for $\zeta$ or
  $\zeta^h$. Fix $i$ and call $J$ the real interval with extreme
  points $x$ and $x-\lambda_i \theta$.
  \begin{eqnarray*}
    &
    & \norma{\int_J  \left(\d{g} - \d{g^h}\right)}
    \\
    & =
    & \left\|
      \sum_{\bar x \in \mathcal{I} (\zeta) \cap J}
      \left(
      \Xi\left(\zeta (\bar x+), \zeta (\bar x), u (\tau, \xi)\right)
      -
      f\left(\zeta (\bar x+),u (\tau, \xi)\right)
      +
      f\left(\zeta (\bar x),u (\tau, \xi)\right)\right)
      \right.
    \\
    &
    & + \int_J
      \left(D_{v (x)}^+ \Xi\left(\zeta (x), \zeta (x), u (\tau,\xi)\right)
      - D_z f\left(\zeta(x),u(\tau,\xi)\right) v (x)\right) \d{\norma{\mu} (x)}
    \\
    &
    & \left.
      - \sum_{\bar x \in \mathcal{I}(\zeta^h) \cap J}
      \left(
      \Xi\left(\zeta^h (\bar x+), \zeta^h (\bar x), u (\tau, \xi)\right)
      -
      f\left(\zeta^h (\bar x+),u (\tau, \xi)\right)
      +
      f\left(\zeta^h (\bar x),u (\tau, \xi)\right)\right)
      \right\|
    \\
    & \leq
    & \mathcal{E}_1 + \mathcal{E}_2 + \mathcal{E}_3 + \mathcal{E}_4 +
      \mathcal{E}_5 + \mathcal{E}_6 \,.
  \end{eqnarray*}
  The terms $\mathcal{E}_1$, $\ldots$, $\mathcal{E}_6$ are defined
  below.

  In the first term, using the definition of $\mathcal{I}^h$
  in~\textit{\ref{item:r2}}, we show that the sum of all jumps in
  $\zeta$ not in $\mathcal{I}^h$ is $\O \, h$:
  \begin{displaymath}
    \begin{array}{rcl}
      \mathcal{E}_1
      &  =
      & \displaystyle
        \norma{
        \sum_{\bar x \in (\mathcal{I} (\zeta) \setminus \mathcal{I}^h) \cap J}
        \left(
        \Xi\left(\zeta (\bar x+), \zeta (\bar x), u (\tau, \xi)\right)
        -
        f\left(\zeta (\bar x+),u (\tau, \xi)\right)
        +
        f\left(\zeta (\bar x),u (\tau, \xi)\right)\right)
        }
      \\
      & \leq
      & \displaystyle
        \O \sum_{\bar x \in \mathcal{I} (\zeta) \setminus \mathcal{I}^h}
        \norma{\Delta \zeta (\bar x)}
        \hfill
        \mbox{[By~\ref{it:f1}, \ref{eq:Xi1} and~\ref{eq:Xi3}]}
      \\
      & \leq
      & \displaystyle
        \O \, h
        \hfill
        \mbox{[By~\ref{item:r2}]}
      \\
      & \to
      & 0 \quad \mbox{ as } h \to 0 \,.
    \end{array}
  \end{displaymath}

  Now we estimate the effect of passing from $\zeta^h (\bar x)$ to
  $\zeta (\bar x)$ in the jumps $\bar x$ in $\mathcal{I}^h$, calling
  $\bar{\bar{x}}$ the point in $\mathcal{I}(\zeta^h)$ that precedes
  $\bar x$ and using~\eqref{eq:51}:
  \begin{displaymath}
    \begin{array}{rcl}
      \mathcal{E}_2
      &  =
      & \displaystyle
        \left\|
        \sum_{\bar x \in \mathcal{I} (\zeta) \cap \mathcal{I}^h \cap J}
        \left[
        \Xi\left(\zeta^h (\bar x+), \zeta (\bar x), u (\tau, \xi)\right)
        -
        f\left(\zeta^h (\bar x+),u (\tau, \xi)\right)
        +
        f\left(\zeta (\bar x),u (\tau, \xi)\right)
        \right.
        \right.
      \\
      &
      & \displaystyle
        \left.
        \quad  \left. -
        \Xi\left(\zeta^h (\bar x+), \zeta^h (\bar x), u (\tau, \xi)\right)
        +
        f\left(\zeta^h (\bar x+),u (\tau, \xi)\right)
        -
        f\left(\zeta^h (\bar x),u (\tau, \xi)\right)
        \right]
        \right\|
      \\
      & \leq
      & \displaystyle
        \O \sum_{\bar x \in \mathcal{I} (\zeta) \cap \mathcal{I}^h \cap J}
        \norma{\zeta (\bar x) - \zeta^h (\bar x)}
        \hfill
        \mbox{[By~\ref{it:f1} and~\ref{eq:Xi1}]}
      \\
      & \leq
      & \displaystyle
        \O \sum_{\bar x \in \mathcal{I} (\zeta) \cap \mathcal{I}^h \cap J}
        \norma{\zeta (\bar x) - \zeta(\bar{\bar x}+)}
        \hfill
        \mbox{[By~\eqref{eq:51}]}
      \\
      & \leq
      & \displaystyle
        \O \sum_{\bar x \in \mathcal{I} (\zeta) \cap \mathcal{I}^h \cap J}
        \tv\left(\zeta; \mathopen]\bar{\bar x}, \bar x\mathclose[\right)
        \hfill
        \mbox{[By~\ref{eq:24}]}
      \\
      & \leq
      & \displaystyle
        \O \sum_{\bar x \in \mathcal{I} (\zeta) \cap \mathcal{I}^h \cap J}
        \dfrac{h}{1+\sharp \mathcal{I}^h}
        \hfill
        \mbox{[By~\ref{item:r3}]}
      \\
      & \leq
      & \O \, h
      \\
      & \to
      & 0 \quad \mbox{ as } h \to 0 \,.
    \end{array}
  \end{displaymath}

  Call $\bar{\bar{x}}$ the point in $\mathcal{I}(\zeta^h)$ that
  precedes $\bar x$.  Out of $\mathcal{I}^h$, the measure $\mu$
  approximates the measure $D \zeta$, so that
  $\zeta^h (\bar x) + \bar\delta \, \bar v$ approximates
  $\zeta^h (\bar x+)$ as in~\eqref{eq:62}--\eqref{eq:61}:
  \begin{eqnarray*}
    \mathcal{E}_3
    &  =
    & \left\|
      \sum_{\bar x \in \mathcal{I} (\zeta^h) \cap J \setminus \mathcal{I}^h}
      \left[
      \Xi\left(\zeta^h (\bar x+), \zeta^h (\bar x), u (\tau, \xi)\right)
      -
      f\left(\zeta^h (\bar x+),u (\tau, \xi)\right)
      +
      f\left(\zeta^h (\bar x),u (\tau, \xi)\right)
      \right.
      \right.
    \\
    &
    & \qquad -
      \left.
      \left.
      \Xi\left(\zeta^h (\bar x)+\bar\delta \bar v, \zeta^h (\bar x), u (\tau, \xi)\right)
      +
      f\left(\zeta^h (\bar x)+\bar\delta \bar v,u (\tau, \xi)\right)
      -
      f\left(\zeta^h (\bar x),u (\tau, \xi)\right)
      \right]
      \right\|
    \\
    & \leq
    & \O
      \sum_{\bar x \in \mathcal{I} (\zeta^h) \cap J \setminus \mathcal{I}^h}
      \norma{
      \zeta^h (\bar x+)
      -
      \zeta^h (\bar x)
      -
      \mu \left(\mathopen]\bar{\bar x}, \bar x \mathclose[ \right)
      }
    \\
    & =
    & \O
      \sum_{\bar x \in \mathcal{I} (\zeta^h) \cap J \setminus \mathcal{I}^h}
      \norma{
      \zeta^h (\bar x+)
      -
      \zeta^h (\bar{\bar x}+)
      -
      \mu \left(\mathopen]\bar{\bar x}, \bar x \mathclose[ \right)
      }
    \\
    & =
    & \O
      \sum_{\bar x \in \mathcal{I} (\zeta^h) \cap J \setminus \mathcal{I}^h}
      \norma{
      D\zeta\left(\mathopen]\bar{\bar x}, \bar x \mathclose] \right)
      -
      \mu \left([\bar{\bar x}, \bar x ] \right)
      }
    \\
    & \leq
    & \O
      \sum_{\bar x \not\in \mathcal{I}^h}
      \norma{\Delta\zeta (\bar x)}
    \\
    & \leq
    & \O \, h
    \\
    & \to
    & 0 \quad \mbox{ as } h \to 0 \,.
  \end{eqnarray*}

  Using Lemma~\ref{lem:uno}, the differences at the jumps in $\zeta^h$
  out of $\mathcal{I}^h$ are approximated by means of derivatives:
  \begin{eqnarray*}
    \mathcal{E}_4
    &  =
    & \left\|
      \sum_{\bar x \in \mathcal{I} (\zeta^h) \cap J \setminus \mathcal{I}^h}
      \left[
      \Xi\left(\zeta^h (\bar x) + \bar \delta \bar v, \zeta^h (\bar x), u (\tau, \xi)\right)
      -
      f\left(\zeta^h (\bar x) + \bar \delta \bar v,u (\tau, \xi)\right)
      +
      f\left(\zeta^h (\bar x),u (\tau, \xi)\right)
      \right.
      \right.
    \\
    &
    &
      \left.
      \left.-
      \bar \delta \left(
      D_{\bar v}^+ \Xi\left(\zeta^h (\bar x), \zeta^h (\bar x), u (\tau,\xi)\right)
      -
      D_z f\left(\zeta^h (\bar x), u (\tau,\xi)\right) \bar v
      \right)
      \right]
      \right\|
    \\
    & \leq
    & \O
      \sum_{\bar x \in \mathcal{I} (\zeta^h) \cap J \setminus \mathcal{I}^h}
      \sigma (\bar \delta) \, \bar \delta
    \\
    & \leq
    & \O \, \sigma (h) \, \tv (\zeta)
    \\
    & \to
    & 0 \quad \mbox{ as } h \to 0 \,.
  \end{eqnarray*}

  If $\bar x \in \mathcal{I}^h$,
  $\norma{\mu}\left(\mathopen]\bar{\bar x}, \bar x \mathclose[\right)
  = \bar \delta$ is negligible, so that by~\ref{item:r3},
  \eqref{eq:60}, \eqref{eq:62} and~\eqref{eq:61},
  \begin{eqnarray*}
    \mathcal{E}_5
    &  =
    & \left\|
      \sum_{\bar x \in \mathcal{I} (\zeta^h) \cap J \setminus \mathcal{I}^h}
      \left[
      \bar \delta \left(
      D_{\bar v} \Xi\left(\zeta^h (\bar x), \zeta^h (\bar x), u (\tau,\xi)\right)
      -
      D_z f\left(\zeta^h (\bar x), u (\tau,\xi)\right) \bar v
      \right)
      \right.
      \right.
    \\
    &
    & -
      \left.
      \left.
      \int_J
      \left(
      D^+_{v^h (x)} \Xi\left(\zeta^h (x),\zeta^h (x), u (\tau,\xi)\right)
      -
      D_z f\left(\zeta^h (x), u (\tau,\xi)\right) v^h (x)
      \right)
      \d{\norma{\mu} (x)}
      \right]
      \right\|
    \\
    & \leq
    & \O \sum_{\bar x \in \mathcal{I}^h} \bar \delta + \O \, h
    \\
    & \leq
    & \O \sum_{\bar x \in \mathcal{I}^h}
      \dfrac{h}{1+\sharp\mathcal{I}^h} + \O \, h
    \\
    & \to
    & 0 \quad \mbox{ as } h \to 0 \,.
  \end{eqnarray*}

  We now use the Lipschitz continuity of $D_v \Xi$, see~\ref{eq:Xi4},
  and of $D_z f$, see~\ref{it:f1}, the uniform convergence of
  $\zeta^h \to \zeta$ and the convergence $v^h \to v$
  by~\eqref{eq:63}:
  \begin{eqnarray*}
    \mathcal{E}_6
    &  =
    & \left\|
      \int_J
      \left(
      D^+_{v^h (x)} \Xi\left(\zeta^h (x),\zeta^h (x), u (\tau,\xi)\right)
      -
      D_z f\left(\zeta^h (x), u (\tau,\xi)\right) v^h (x)
      \right)
      \d{\norma{\mu} (x)}
      \right.
    \\
    &
    & - \left.
      \int_J
      \left(
      D^+_{v (x)} \Xi\left(\zeta (x),\zeta (x), u (\tau,\xi)\right)
      -
      D_z f\left(\zeta (x), u (\tau,\xi)\right) v (x)
      \right)
      \d{\norma{\mu} (x)}
      \right\|
    \\
    & \leq
    & \O \int_J
      \left(
      \norma{\zeta (x) - \zeta^h (x)}
      +
      \norma{v (x) - v^h (x)}
      \right)
      \d{\norma{\mu} (x)}
    \\
    & \to
    & 0 \quad \mbox{ as } h \to 0 \,.
  \end{eqnarray*}
  The proof of~\emph{(ii)} is completed.

  \bigskip

  We now prove that a $\L1$--Lipschitz continuous map $u$
  satisfying~\emph{(i)} and~\emph{(ii)} for a.e.~$t$ is actually an
  orbit of $S^\zeta$. Using~\cite[Theorem~2.9]{Bressan2000} as
  in~\cite[\S~9.2]{Bressan2000}, for any $a,b \in \reali$ with $a<b$
  \begin{eqnarray*}
    &
    & \norma{u (t)-S^\zeta_tu (0)}_{\L1 ([a+\hat\lambda \, t, b-\hat\lambda t]; \reali^n)}
    \\
    & \leq
    & L \, \int_0^t
      \liminf_{h\to 0+} \dfrac{1}{h} \,
      \norma{u (\tau+h) - S^\zeta_h u (\tau)}_{\L1 ([a+\hat\lambda \, (\tau+h), b-\hat\lambda (\tau+h)]; \reali^n)} \d\tau \,.
  \end{eqnarray*}
  Let $\tau$ be such that~\emph{(i)} and~\emph{(ii)} hold. Fix
  $\epsilon>0$ and choose $x_0 = a + \hat\lambda \tau$,
  $x_0 < x_1< x_2 < \ldots < x_{N-1} < x_N$,
  $x_N = b - \hat\lambda\tau$ such that, for $i=1, \ldots, N$,
  \begin{displaymath}
    \tv (u (\tau); \mathopen]x_{i-1}, x_i \mathclose[)
    +
    \tv (\zeta; \mathopen]x_{i-1}, x_i \mathclose[)
    < \epsilon \,.
  \end{displaymath}
  Then, for $h>0$ sufficiently small, and for
  $\xi_i \in \mathopen]x_{i-1}, x_i \mathclose[$
  \begin{eqnarray*}
    &
    &
      \norma{u (\tau+h) - S^\zeta_h u (\tau)}_{\L1 ([a+\hat\lambda \, (\tau+h), b-\hat\lambda (\tau+h)]; \reali^n)}
    \\
    & =
    & \sum_{i=1}^N
      \int_{x_{i-1}+\hat\lambda \, h}^{x_i-\hat\lambda \, h}
      \norma{u (\tau+h,x) - \left(S^\zeta_h u (\tau)\right) (x)} \d{x}
    \\
    &
    &
      +
      \sum_{i=1}^{N-1} \int^{x_i+\hat\lambda \, h}_{x_i-\hat\lambda \, h}
      \norma{u (\tau+h,x) - \left(S^\zeta_h u (\tau)\right) (x)} \d{x}
    \\
    & =
    & \sum_{i=1}^N
      \int_{x_{i-1}+\hat\lambda \, h}^{x_i-\hat\lambda \, h}
      \norma{u (\tau+h,x) - U^\flat_{(u;\tau,\xi_i)} (h,x)} \d{x}
    \\
    &
    & +
      \sum_{i=1}^N
      \int_{x_{i-1}+\hat\lambda \, h}^{x_i-\hat\lambda \, h}
      \norma{U^\flat_{(u;\tau,\xi_i)} (h,x) - \left(S^\zeta_h u (\tau)\right) (x)} \d{x}
    \\
    &
    &
      +
      \sum_{i=1}^{N-1} \int^{x_i+\hat\lambda \, h}_{x_i-\hat\lambda \, h}
      \norma{u (\tau+h,x) - U^\sharp_{(u;\tau,x_i)} (h,x)} \d{x}
    \\
    &
    &
      +
      \sum_{i=1}^{N-1} \int^{x_i+\hat\lambda \, h}_{x_i-\hat\lambda \, h}
      \norma{U^\sharp_{(u;\tau,x_i)} (h,x) - \left(S^\zeta_h u (\tau)\right) (x)} \d{x} \,.
  \end{eqnarray*}
  Since both $u$ and $S^\zeta$ satisfy~\emph{(ii)}, we get
  \begin{eqnarray*}
    &
    & \dfrac{1}{h} \norma{u (\tau+h) - S^\zeta_h u (\tau)}_{\L1 ([a+\hat\lambda \, t, b-\hat\lambda t]; \reali^n)}
    \\
    & \leq
    & \O  \sum_{i=1}^N
      \left(
      \tv\left(u (\tau); \mathopen]x_{i-1}, x_i\mathclose[\right)
      +
      \tv\left(\zeta; \mathopen]x_{i-1}, x_i\mathclose[\right)
      \right)^2
    \\
    &
    & +
      \sum_{i=1}^{N-1} \frac1h \int_{x_i+\hat\lambda \, h}^{x_i-\hat\lambda \, h}
      \norma{u (\tau+h,x) - U^\sharp_{(u;\tau,x_i)} (h,x)} \d{x}
    \\
    &
    &
      +
      \sum_{i=1}^{N-1} \frac1h \int_{x_i+\hat\lambda \, h}^{x_i-\hat\lambda \, h}
      \norma{U^\sharp_{(u;\tau,x_i)} (h,x) - \left(S_h u (\tau)\right) (x)} \d{x} \,.
  \end{eqnarray*}
  Both $u$ and $S^\zeta$ satisfy~\emph{(i)}, hence in the
  $\liminf_{h\to 0}$ the latter two terms vanish. Thus,
  \begin{eqnarray*}
    &
    & \liminf_{h\to 0}  \dfrac{1}{h} \norma{u (\tau+h) - S^\zeta_h u (\tau)}_{\L1 ([a+\hat\lambda \, t, b-\hat\lambda t]; \reali^n)}
    \\
    & \leq
    & \O  \sum_{i=1}^N
      \left(
      \tv\left(u (\tau); \mathopen]x_{i-1}, x_i\mathclose[\right)
      +
      \tv\left(\zeta; \mathopen]x_{i-1}, x_i\mathclose[\right)
      \right)^2
    \\
    & \leq
    &
      \O \, \epsilon \,
      \left(
      \tv\left(u (\tau)\right)
      +
      \tv(\zeta)
      \right) \,.
  \end{eqnarray*}
  Since $\epsilon$ is arbitrary, the term in the left hand side above
  vanishes. The arbitrariness of $a$ and $b$ allows to complete the
  proof.
\end{proofof}

\bigskip

\noindent\textit{Acknowledgment.} The first and second authors
were partly supported by the GNAMPA~2022 project \emph{"Evolution
  Equations: well posedness, control and applications"}. The work of
the third author has been funded by the Deutsche
Forschungsgemeinschaft (DFG, German Research Foundation) Projektnummer
320021702/GRK2326 Energy, Entropy, and Dissipative Dynamics (EDDy).

{ \small

  \bibliography{CGH}

\begin{thebibliography}{10}

\bibitem{AmadoriGosseGuerra2002}
D.~Amadori, L.~Gosse, and G.~Guerra.
\newblock Global {BV} entropy solutions and uniqueness for hyperbolic systems
  of balance laws.
\newblock {\em Arch. Ration. Mech. Anal.}, 162(4):327--366, 2002.

\bibitem{BaitiJenssen1998}
P.~Baiti and H.~K. Jenssen.
\newblock On the front-tracking algorithm.
\newblock {\em J. Math. Anal. Appl.}, 217(2):395--404, 1998.

\bibitem{BianchiniColombo2002}
S.~Bianchini and R.~M. Colombo.
\newblock On the stability of the standard {R}iemann semigroup.
\newblock {\em Proc. Amer. Math. Soc.}, 130(7):1961--1973, 2002.

\bibitem{Bressan}
A.~Bressan.
\newblock The unique limit of the {G}limm scheme.
\newblock {\em Arch. Rational Mech. Anal.}, 130(3):205--230, 1995.

\bibitem{Bressan2000}
A.~Bressan.
\newblock {\em Hyperbolic systems of conservation laws}, volume~20 of {\em
  Oxford Lecture Series in Mathematics and its Applications}.
\newblock Oxford University Press, Oxford, 2000.
\newblock The one-dimensional Cauchy problem.

\bibitem{BressanLiuYang1999}
A.~Bressan, T.-P. Liu, and T.~Yang.
\newblock {$L^1$} stability estimates for {$n\times n$} conservation laws.
\newblock {\em Arch. Ration. Mech. Anal.}, 149(1):1--22, 1999.

\bibitem{ColomboGuerra2008}
R.~M. Colombo and G.~Guerra.
\newblock On the stability functional for conservation laws.
\newblock {\em Nonlinear Anal.}, 69(5-6):1581--1598, 2008.

\bibitem{ColomboGuerraHolle}
R.~M. Colombo, G.~Guerra, and Y.~Holle.
\newblock Non conservative products in fluid dynamics.
\newblock {\em Nonlinear Analysis: Real World Applications}, 2022.

\bibitem{ColomboMarcellini2010}
R.~M. Colombo and F.~Marcellini.
\newblock Smooth and discontinuous junctions in the {$p$}-system.
\newblock {\em J. Math. Anal. Appl.}, 361(2):440--456, 2010.

\bibitem{DafermosWFT}
C.~M. Dafermos.
\newblock Polygonal approximations of solutions of the initial value problem
  for a conservation law.
\newblock {\em J. Math. Anal. Appl.}, 38:33--41, 1972.

\bibitem{Dafermos2000}
C.~M. Dafermos.
\newblock {\em Hyperbolic conservation laws in continuum physics}, volume 325
  of {\em Grundlehren der Mathematischen Wissenschaften [Fundamental Principles
  of Mathematical Sciences]}.
\newblock Springer-Verlag, Berlin, fourth edition, 2016.

\bibitem{DalMasoMurat1995}
G.~Dal~Maso, P.~G. Lefloch, and F.~Murat.
\newblock Definition and weak stability of nonconservative products.
\newblock {\em J. Math. Pures Appl. (9)}, 74(6):483--548, 1995.

\bibitem{GuerraMarcelliniSchleper2009}
G.~Guerra, F.~Marcellini, and V.~Schleper.
\newblock Balance laws with integrable unbounded sources.
\newblock {\em SIAM J. Math. Anal.}, 41(3):1164--1189, 2009.

\bibitem{MR1912206}
H.~Holden and N.~H. Risebro.
\newblock {\em Front tracking for hyperbolic conservation laws}, volume 152 of
  {\em Applied Mathematical Sciences}.
\newblock Springer-Verlag, New York, 2002.

\bibitem{MR2177892}
D.~Kr\"{o}ner and M.~D. Thanh.
\newblock Numerical solutions to compressible flows in a nozzle with variable
  cross-section.
\newblock {\em SIAM J. Numer. Anal.}, 43(2):796--824, 2005.

\bibitem{Lax1957}
P.~D. Lax.
\newblock Hyperbolic systems of conservation laws. {II}.
\newblock {\em Comm. Pure Appl. Math.}, 10:537--566, 1957.

\bibitem{MR2041456}
P.~G. Lefloch and M.~D. Thanh.
\newblock The {R}iemann problem for fluid flows in a nozzle with discontinuous
  cross-section.
\newblock {\em Commun. Math. Sci.}, 1(4):763--797, 2003.

\bibitem{MR1718304}
P.~G. Lefloch and A.~E. Tzavaras.
\newblock Representation of weak limits and definition of nonconservative
  products.
\newblock {\em SIAM J. Math. Anal.}, 30(6):1309--1342, 1999.

\bibitem{MR4179858}
X.~Liu.
\newblock A well-balanced and positivity-preserving numerical model for shallow
  water flows in channels with wet-dry fronts.
\newblock {\em J. Sci. Comput.}, 85(3):Paper No. 60, 22, 2020.

\bibitem{Rudin}
W.~Rudin.
\newblock {\em Real and complex analysis}.
\newblock McGraw-Hill Book Co., New York, third edition, 1987.

\bibitem{Yong1999}
W.-A. Yong.
\newblock A simple approach to {G}limm's interaction estimates.
\newblock {\em Appl. Math. Lett.}, 12(2):29--34, 1999.

\end{thebibliography}

  \bibliographystyle{abbrv}

}
\end{document}